\documentclass[11pt]{amsart}

\DeclareMathOperator{\Lie}{Lie}

\DeclareMathOperator{\ad}{ad}

\usepackage[usenames,dvipsnames]{color}

\usepackage{amsmath,amssymb,amsthm,graphics,amscd}
\usepackage{longtable,pifont,multirow}
\usepackage[margin={0.1\paperwidth,0.1\paperheight},footskip=.3in]{geometry}
\usepackage{cite,xypic}
\usepackage{color,enumerate}
\usepackage{graphicx}
\usepackage{epsf,multicol}
\usepackage{fancyhdr,hyperref}
\usepackage{lscape,calc}
\usepackage[bf,calcwidth=.9\linewidth,hang]{caption}

\hypersetup{backref,colorlinks=true,citecolor=Purple}

  \renewenvironment{thebibliography}[1]{
    \begin{oldthebibliography}{#1}
      \setlength{\parskip}{0ex}
      \setlength{\itemsep}{0ex}
  }
  {
    \end{oldthebibliography}
  }
  
\setlongtables

\newcounter{rownum}
\newtheorem{lemma}{Lemma}[section]
\newtheorem{theorem}[lemma]{Theorem}

\theoremstyle{definition}

\newtheorem*{algorithm*}{Algorithm}
\theoremstyle{remark}
\newtheorem{remark}[lemma]{Remark}
\newtheorem{remarks}[lemma]{Remarks}

\theoremstyle{definition}
\newtheorem{defn}[lemma]{Definition}

\newcommand{\Ext}{\mathrm{Ext}}

\newcommand{\SO}{\mathrm{SO}}
\newcommand{\Sp}{\mathrm{Sp}}
\newcommand{\SL}{\mathrm{SL}}

\newcommand{\diag}{\mathrm{diag}}

\newcommand{\Z}{\mathbb{Z}}

\newcommand{\gl}{\mathfrak{gl}}

\newcommand{\m}{\mathfrak{m}}

\newcommand{\h}{\mathfrak{h}}

\newcommand{\so}{\mathfrak{so}}

\renewcommand{\sp}{\mathfrak{sp}}

\newcommand{\Vm}{V_\text{min}}

\def\arraystretch{1.2}

\subjclass[2010]{17B45, 20G99}
\title{Representatives for unipotent classes and nilpotent orbits}
\author{Mikko Korhonen${}^*{}^\dagger$}
\thanks{$^*$ Partially supported by SNSF grant 200021\_146223 and NSFC grants 11771200 and 11931005}
\thanks{$^\dagger$ We thank the Isaac Newton Institute for Mathematical Sciences for support and hospitality during the programme Groups, Representations and Applications: New perspectives, when some of the work on this paper was undertaken. (EPSRC grant number EP/R014604/1.)}
\address{\parbox{\linewidth}{Department of Mathematics, Southern University of Science and Technology, \text{Shenzhen} 518055,\\ \text{Guangdong}, P. R. China}}
\email{korhonen\_mikko@hotmail.com {\text{\rm(Korhonen)}}}

\author{David I. Stewart$^\dagger$}
\address{School of Mathematics and Statistics, Herschel Building, Newcastle, NE1 7RU, UK
} 
\email{david.stewart@ncl.ac.uk {\text{\rm(Stewart)}}}

\author{Adam R. Thomas$^\dagger$} 
\address{Department of Mathematics, University of Warwick, Coventry, CV4 7AL, UK} 
\email{adam.r.thomas@warwick.ac.uk {\text{\rm{(Thomas)}}}}

\newlength\oldparskip 
\begin{document}
\pagestyle{plain}
\begin{abstract} Let $G$ be a simple algebraic group over an algebraically closed field $k$ of characteristic $p$. The classification of the conjugacy classes of unipotent elements of $G(k)$ and nilpotent orbits of $G$ on $\Lie(G)$ is well-established. One knows there are representatives of every unipotent class as a product of root group elements and every nilpotent orbit as a sum of root elements. We give explicit representatives in terms of a Chevalley basis for the {\em eminent} classes. A unipotent (resp. nilpotent) element is said to be eminent if it is not contained in any subsystem subgroup (resp. subalgebra), or a natural generalisation if $G$ is of type $D_n$. From these representatives, it is straightforward to generate representatives for any given class. Along the way we also prove recognition theorems for identifying both the unipotent classes and nilpotent orbits of exceptional algebraic groups.     
\end{abstract}

\vspace*{-8ex}
\maketitle

\vspace*{-4ex}
\section{Introduction}\label{s:intro}
Let $G$ be a connected reductive algebraic group over an algebraically closed field $k$ of characteristic $p$. Fix a maximal torus $T$ of $G$, and positive system of roots; then it is well-known that every unipotent class has a representative which is a product of root group elements and that every nilpotent orbit has a representative which is a sum of root elements. However, it is not easy to extract such representatives from the existing literature and there are a few errors (some perpetrated by the second author). This note gives a definitive list when $G$ is simple. For compactness, we provide the representatives of certain distinguished classes we call {\em eminent} classes.  An eminent unipotent (resp.~nilpotent) element is one which is not contained in any proper subsystem subgroup (resp. subalgebra) of $G$, or a natural generalisation if $G$ has a factor of type $D_n$ ---the precise definition is given in the next section. Given representatives of the eminent classes in the simple case, it is then a routine task to find representatives of all classes for $G$. (See Section \ref{s:allclasses}.)

\begin{theorem}\label{nilreps} Let $G$ be a simple $k$-group over a field $k=\bar k$ of characteristic $p$ with Lie algebra $\Lie(G)$. 
\begin{enumerate}[\normalfont (i)]
\item A nilpotent element $e\in \Lie(G)$ is eminent if and only if it is $G$-conjugate to one of the representatives in Table \ref{t:classnil} ($G$ of classical type) or Table~\ref{t:exnil} ($G$ of exceptional type). Moreover, no two representatives in Tables \ref{t:classnil} and \ref{t:exnil} are $G$-conjugate.
\item A unipotent element $u\in G$ is eminent if and only if it is $G$-conjugate to one of the representatives in Table \ref{t:classunip} ($G$ of classical type) or Table~\ref{t:exunip} ($G$ of exceptional type). Moreover, no two representatives in Tables \ref{t:classunip} and \ref{t:exunip} are $G$-conjugate.
\end{enumerate}
\end{theorem}

We note that the only time a regular element is not eminent is when $G$ is of type $D_n$. Indeed, such elements are contained in generalised subsystem subgroups or subalgebras of type $B_{n-1}$.

We make a few remarks on the tables. An explanation of the notation used is given in Section \ref{s:notation}; we observe Bourbaki doctrine when exhibiting representatives. When $G=\SL(V)$, $\SO(V)$ or $\Sp(V)$ is a classical group, we also supply the action of the elements on the natural $G$-module $V$, which characterises them up to conjugacy in $G$. Lastly, since we use the same representatives as \cite[Tables~13.3, 14.1]{LS12}, one can invoke \cite[Thms~19.1,19.2]{LS12} to generate representatives for the classes in Table \ref{t:exunip} from those in Table \ref{t:exnil} as follows: a representative of the unipotent class with label $X$ is $\prod_{j \in J} x_{\alpha_j}(1)$, where $\sum_{j \in J} e_{\alpha_j}$ is the representative of the nilpotent element with label $X$. 
\setlength{\parskip}{\oldparskip}
\begin{center}\begin{longtable}{|c|c|c|c|}
\hline 
Type & Char. & Representation & Representative\\
\hline
$A_n$ & any & $V(n+1)$ & $e_{\alpha_1}+\dots+e_{\alpha_n}$\\
\hline
$B_n$; $n \geq 2$ & $p \neq 2$ & $V(2n+1)$ & $e_{\alpha_1}+\dots+e_{\alpha_n}$\\ 
& $p=2$ & $D(n+1)$ & $e_{\alpha_1}+\dots+e_{\alpha_n}$\\ 
\hline
$C_n$; $n\geq 2$ & any & $V(2n)$ & $e_{\alpha_1}+\dots+e_{\alpha_n}$\\
 & $p=2$ & $W_l(n)$; $0<l<\frac{n}{2}$  & $\left(\sum_{i=1}^{n-1} e_{\alpha_i}\right)+  e_{2 \alpha_l + \cdots + 2 \alpha_{n-1} + \alpha_n}$\\
\hline
$D_n$; $n\geq 4$ & $p=2$ & $W_l(n)$; $\frac{n+1}{2}<l< n$ & $\sum_{i=1}^{n-1}e_{\alpha_i}+e_{\beta_l}$, where $\beta_l = \alpha_{n} + \sum_{i=2l-n-1}^{n-2} \alpha_{i}$\\
\hline
\caption{Eminent nilpotent representatives in classical types}\label{t:classnil}\end{longtable}\end{center}

\let\oldarraystretch\arraystretch
\let\oldarraycolsep\arraycolsep

\begin{center}\begin{longtable}{|c|c|c|c|}\hline Type & Char. & Label & Representative\\\hline
$G_2$ & any & $G_2$ &  \def\arraystretch{0.5} \arraycolsep=0pt $e_{\tiny\begin{array}{c c c c c c c}1&0\end{array}}+ e_{\tiny\begin{array}{c c c c c c c}0&1\end{array}}$ \\
& $p=3$ & $(\tilde{A}_1)_3$ & \def\arraystretch{0.5} \arraycolsep=0pt $e_{\tiny\begin{array}{c c c c c c c}2&1\end{array}}+
e_{\tiny\begin{array}{c c c c c c c}3&2\end{array}}$ \\ \hline 
$F_4$ & any & $F_4$ & \def\arraystretch{0.5} \arraycolsep=0pt $e_{\tiny\begin{array}{c c c c c c c}1&0&0&0\end{array}}+
e_{\tiny\begin{array}{c c c c c c c}0&1&0&0\end{array}}+
e_{\tiny\begin{array}{c c c c c c c}0&0&1&0\end{array}}+
e_{\tiny\begin{array}{c c c c c c c}0&0&0&1\end{array}}$ \\
& $p=2$ & $F_4(a_2)$ & \def\arraystretch{0.5} \arraycolsep=0pt $e_{\tiny\begin{array}{c c c c c c c}0&0&0&1\end{array}}+
e_{\tiny\begin{array}{c c c c c c c}0&0&1&1\end{array}}+
e_{\tiny\begin{array}{c c c c c c c}1&1&0&0\end{array}}+
e_{\tiny\begin{array}{c c c c c c c}0&1&2&0\end{array}}$ \\
& $p=2$ & $(C_3)_2$ & \def\arraystretch{0.5} \arraycolsep=0pt $e_{\tiny\begin{array}{c c c c c c c}0&0&0&1\end{array}}+
e_{\tiny\begin{array}{c c c c c c c}1&1&1&0\end{array}}+
e_{\tiny\begin{array}{c c c c c c c}0&1&2&0\end{array}}+
e_{\tiny\begin{array}{c c c c c c c}1&2&2&2\end{array}}$ \\
& $p=2$ & $(\tilde{A}_2 A_1)_2$ & \def\arraystretch{0.5} \arraycolsep=0pt $e_{\tiny\begin{array}{c c c c c c c}0&1&1&1\end{array}}+
e_{\tiny\begin{array}{c c c c c c c}1&1&2&1\end{array}}+
e_{\tiny\begin{array}{c c c c c c c}0&1&2&2\end{array}}+
e_{\tiny\begin{array}{c c c c c c c}1&2&2&0\end{array}}$ \\
\hline

$E_6$ & any & $E_6$ & \def\arraystretch{0.5} \arraycolsep=0pt $e_{\tiny\begin{array}{c c c c c c c}1&0&0&0&0\\&&0\end{array}}+
e_{\tiny\begin{array}{c c c c c c c}0&0&0&0&0\\&&1\end{array}}+
e_{\tiny\begin{array}{c c c c c c c}0&1&0&0&0\\&&0\end{array}}+
e_{\tiny\begin{array}{c c c c c c c}0&0&1&0&0\\&&0\end{array}}+
e_{\tiny\begin{array}{c c c c c c c}0&0&0&1&0\\&&0\end{array}}+
e_{\tiny\begin{array}{c c c c c c c}0&0&0&0&1\\&&0\end{array}}$ \\
& any & $E_6(a_1)$ & \def\arraystretch{0.5} \arraycolsep=0pt $e_{\tiny\begin{array}{c c c c c c c}1&0&0&0&0\\&&0\end{array}}+
e_{\tiny\begin{array}{c c c c c c c}0&1&0&0&0\\&&0\end{array}}+
e_{\tiny\begin{array}{c c c c c c c}0&0&0&1&0\\&&0\end{array}}+
e_{\tiny\begin{array}{c c c c c c c}0&0&0&0&1\\&&0\end{array}}+
e_{\tiny\begin{array}{c c c c c c c}0&0&1&0&0\\&&1\end{array}}+
e_{\tiny\begin{array}{c c c c c c c}0&1&1&0&0\\&&0\end{array}}$\\
& $p=2$ & $E_6(a_3)$ & \def\arraystretch{0.5} \arraycolsep=0pt $e_{\tiny\begin{array}{c c c c c c c}1&0&0&0&0\\&&0\end{array}}+
e_{\tiny\begin{array}{c c c c c c c}0&0&1&0&0\\&&1\end{array}}+
e_{\tiny\begin{array}{c c c c c c c}0&1&1&0&0\\&&0\end{array}}+
e_{\tiny\begin{array}{c c c c c c c}0&0&0&1&1\\&&0\end{array}}+
e_{\tiny\begin{array}{c c c c c c c}0&0&1&1&0\\&&1\end{array}}+
e_{\tiny\begin{array}{c c c c c c c}0&1&1&1&0\\&&1\end{array}}$\\

\hline
$E_7$ & any & $E_7$ & \def\arraystretch{0.5} \arraycolsep=0pt $e_{\tiny\begin{array}{c c c c c c c}1&0&0&0&0&0\\&&0\end{array}}+
e_{\tiny\begin{array}{c c c c c c c}0&0&0&0&0&0\\&&1\end{array}}+
e_{\tiny\begin{array}{c c c c c c c}0&1&0&0&0&0\\&&0\end{array}}+
e_{\tiny\begin{array}{c c c c c c c}0&0&1&0&0&0\\&&0\end{array}}+
e_{\tiny\begin{array}{c c c c c c c}0&0&0&1&0&0\\&&0\end{array}}+
e_{\tiny\begin{array}{c c c c c c c}0&0&0&0&1&0\\&&0\end{array}}+
e_{\tiny\begin{array}{c c c c c c c}0&0&0&0&0&1\\&&0\end{array}}$ \\
& any & $E_7(a_1)$ & \def\arraystretch{0.5} \arraycolsep=0pt $e_{\tiny\begin{array}{c c c c c c c}1&0&0&0&0&0\\&&0\end{array}}+
e_{\tiny\begin{array}{c c c c c c c}0&1&0&0&0&0\\&&0\end{array}}+
e_{\tiny\begin{array}{c c c c c c c}0&0&0&1&0&0\\&&0\end{array}}+
e_{\tiny\begin{array}{c c c c c c c}0&0&0&0&1&0\\&&0\end{array}}+
e_{\tiny\begin{array}{c c c c c c c}0&0&0&0&0&1\\&&0\end{array}}+
e_{\tiny\begin{array}{c c c c c c c}0&0&1&0&0&0\\&&1\end{array}}+
e_{\tiny\begin{array}{c c c c c c c}0&1&1&0&0&0\\&&0\end{array}}$ \\
& any & $E_7(a_2)$ & \def\arraystretch{0.5} \arraycolsep=0pt $e_{\tiny\begin{array}{c c c c c c c}1&0&0&0&0&0\\&&0\end{array}}+
e_{\tiny\begin{array}{c c c c c c c}0&0&0&0&0&0\\&&1\end{array}}+
e_{\tiny\begin{array}{c c c c c c c}0&1&0&0&0&0\\&&0\end{array}}+
e_{\tiny\begin{array}{c c c c c c c}0&0&1&0&0&0\\&&1\end{array}}+
e_{\tiny\begin{array}{c c c c c c c}0&0&1&1&0&0\\&&0\end{array}}+
e_{\tiny\begin{array}{c c c c c c c}0&0&0&1&1&0\\&&0\end{array}}+
e_{\tiny\begin{array}{c c c c c c c}0&0&0&0&1&1\\&&0\end{array}}$ \\
& $p=2$ & $E_7(a_3)$ & \def\arraystretch{0.5} \arraycolsep=0pt $e_{\tiny\begin{array}{c c c c c c c}1&0&0&0&0&0\\&&0\end{array}}+
e_{\tiny\begin{array}{c c c c c c c}0&0&0&0&0&1\\&&0\end{array}}+
e_{\tiny\begin{array}{c c c c c c c}0&0&1&0&0&0\\&&1\end{array}}+
e_{\tiny\begin{array}{c c c c c c c}0&1&1&0&0&0\\&&0\end{array}}+
e_{\tiny\begin{array}{c c c c c c c}0&0&0&1&1&0\\&&0\end{array}}+
e_{\tiny\begin{array}{c c c c c c c}0&0&1&1&0&0\\&&1\end{array}}+
e_{\tiny\begin{array}{c c c c c c c}0&1&1&1&0&0\\&&1\end{array}}$ \\
& $p=2$ & $(A_6)_2$ & \def\arraystretch{0.5} \arraycolsep=0pt $e_{\tiny\begin{array}{c c c c c c c}0&0&0&1&1&0\\&&0\end{array}}+
e_{\tiny\begin{array}{c c c c c c c}0&0&0&0&1&1\\&&0\end{array}}+
e_{\tiny\begin{array}{c c c c c c c}1&1&1&0&0&0\\&&0\end{array}}+
e_{\tiny\begin{array}{c c c c c c c}0&1&1&0&0&0\\&&1\end{array}}+
e_{\tiny\begin{array}{c c c c c c c}0&0&1&1&0&0\\&&1\end{array}}+
e_{\tiny\begin{array}{c c c c c c c}0&1&1&1&0&0\\&&0\end{array}}+
e_{\tiny\begin{array}{c c c c c c c}1&2&2&1&0&0\\&&1\end{array}}$ \\
\hline
$E_8$ & any & $E_8$ & \def\arraystretch{0.5} \arraycolsep=0pt $e_{\tiny\begin{array}{c c c c c c c c}1&0&0&0&0&0&0\\&&0\end{array}}+
e_{\tiny\begin{array}{c c c c c c c c}0&0&0&0&0&0&0\\&&1\end{array}}+
e_{\tiny\begin{array}{c c c c c c c c}0&1&0&0&0&0&0\\&&0\end{array}}+
e_{\tiny\begin{array}{c c c c c c c c}0&0&1&0&0&0&0\\&&0\end{array}}+
e_{\tiny\begin{array}{c c c c c c c c}0&0&0&1&0&0&0\\&&0\end{array}}+
e_{\tiny\begin{array}{c c c c c c c c}0&0&0&0&1&0&0\\&&0\end{array}}+
e_{\tiny\begin{array}{c c c c c c c c}0&0&0&0&0&1&0\\&&0\end{array}}+
e_{\tiny\begin{array}{c c c c c c c c}0&0&0&0&0&0&1\\&&0\end{array}}$ \\
& any & $E_8(a_1)$ & \def\arraystretch{0.5} \arraycolsep=0pt $e_{\tiny\begin{array}{c c c c c c c c}1&0&0&0&0&0&0\\&&0\end{array}}+
e_{\tiny\begin{array}{c c c c c c c c}0&0&0&0&0&0&0\\&&1\end{array}}+
e_{\tiny\begin{array}{c c c c c c c c}0&0&0&1&0&0&0\\&&0\end{array}}+
e_{\tiny\begin{array}{c c c c c c c c}0&0&0&0&1&0&0\\&&0\end{array}}+
e_{\tiny\begin{array}{c c c c c c c c}0&0&0&0&0&1&0\\&&0\end{array}}+
e_{\tiny\begin{array}{c c c c c c c c}0&0&0&0&0&0&1\\&&0\end{array}}+
e_{\tiny\begin{array}{c c c c c c c c}0&0&1&0&0&0&0\\&&1\end{array}}+
e_{\tiny\begin{array}{c c c c c c c c}0&1&1&0&0&0&0\\&&0\end{array}}$ \\
& any & $E_8(a_2)$ & \def\arraystretch{0.5} \arraycolsep=0pt $e_{\tiny\begin{array}{c c c c c c c c}1&0&0&0&0&0&0\\&&0\end{array}}+
e_{\tiny\begin{array}{c c c c c c c c}0&0&0&0&0&0&0\\&&1\end{array}}+
e_{\tiny\begin{array}{c c c c c c c c}0&1&0&0&0&0&0\\&&0\end{array}}+
e_{\tiny\begin{array}{c c c c c c c c}0&0&0&0&0&0&1\\&&0\end{array}}+
e_{\tiny\begin{array}{c c c c c c c c}0&0&1&0&0&0&0\\&&1\end{array}}+
e_{\tiny\begin{array}{c c c c c c c c}0&0&1&1&0&0&0\\&&0\end{array}}+
e_{\tiny\begin{array}{c c c c c c c c}0&0&0&1&1&0&0\\&&0\end{array}}+
e_{\tiny\begin{array}{c c c c c c c c}0&0&0&0&1&1&0\\&&0\end{array}}$ \\
& $p=2$ & $E_8(a_3)$ & \def\arraystretch{0.5} \arraycolsep=0pt $e_{\tiny\begin{array}{c c c c c c c c}0&0&0&0&0&1&0\\&&0\end{array}}+
e_{\tiny\begin{array}{c c c c c c c c}0&0&0&0&0&0&1\\&&0\end{array}}+
e_{\tiny\begin{array}{c c c c c c c c}1&1&0&0&0&0&0\\&&0\end{array}}+
e_{\tiny\begin{array}{c c c c c c c c}0&0&1&0&0&0&0\\&&1\end{array}}+
e_{\tiny\begin{array}{c c c c c c c c}0&1&1&0&0&0&0\\&&0\end{array}}+
e_{\tiny\begin{array}{c c c c c c c c}0&0&1&1&0&0&0\\&&0\end{array}}+
e_{\tiny\begin{array}{c c c c c c c c}0&0&0&1&1&0&0\\&&0\end{array}}+
e_{\tiny\begin{array}{c c c c c c c c}0&1&1&1&0&0&0\\&&0\end{array}}$ \\
& $p=2$ & $E_8(a_4)$ & \def\arraystretch{0.5} \arraycolsep=0pt $e_{\tiny\begin{array}{c c c c c c c c}1&1&0&0&0&0&0\\&&0\end{array}}+
e_{\tiny\begin{array}{c c c c c c c c}0&0&1&0&0&0&0\\&&1\end{array}}+
e_{\tiny\begin{array}{c c c c c c c c}0&1&1&0&0&0&0\\&&0\end{array}}+
e_{\tiny\begin{array}{c c c c c c c c}0&0&1&1&0&0&0\\&&0\end{array}}+
e_{\tiny\begin{array}{c c c c c c c c}0&0&0&1&1&0&0\\&&0\end{array}}+
e_{\tiny\begin{array}{c c c c c c c c}0&0&0&0&1&1&0\\&&0\end{array}}+
e_{\tiny\begin{array}{c c c c c c c c}0&0&0&0&0&1&1\\&&0\end{array}}+
e_{\tiny\begin{array}{c c c c c c c c}0&1&1&1&0&0&0\\&&0\end{array}}$ \\
& $p=2$ & $E_8(a_5)$ & \def\arraystretch{0.5} \arraycolsep=0pt $e_{\tiny\begin{array}{c c c c c c c c}1&1&0&0&0&0&0\\&&0\end{array}}+
e_{\tiny\begin{array}{c c c c c c c c}0&0&0&0&1&1&0\\&&0\end{array}}+
e_{\tiny\begin{array}{c c c c c c c c}0&0&0&0&0&1&1\\&&0\end{array}}+
e_{\tiny\begin{array}{c c c c c c c c}0&1&1&0&0&0&0\\&&1\end{array}}+
e_{\tiny\begin{array}{c c c c c c c c}0&0&1&1&0&0&0\\&&1\end{array}}+
e_{\tiny\begin{array}{c c c c c c c c}0&1&1&1&0&0&0\\&&0\end{array}}+
e_{\tiny\begin{array}{c c c c c c c c}0&0&1&1&1&0&0\\&&0\end{array}}+
e_{\tiny\begin{array}{c c c c c c c c}0&0&1&1&1&0&0\\&&1\end{array}}$\\
& $p=2$ & $(D_7(a_1))_2$ & \def\arraystretch{0.5} \arraycolsep=0pt $e_{\tiny\begin{array}{c c c c c c c c}0&0&0&1&0&0&0\\&&0\end{array}}+
e_{\tiny\begin{array}{c c c c c c c c}0&0&0&0&0&0&1\\&&0\end{array}}+
e_{\tiny\begin{array}{c c c c c c c c}1&1&0&0&0&0&0\\&&0\end{array}}+
e_{\tiny\begin{array}{c c c c c c c c}0&0&1&1&0&0&0\\&&0\end{array}}+
e_{\tiny\begin{array}{c c c c c c c c}0&0&0&0&0&1&1\\&&0\end{array}}+
e_{\tiny\begin{array}{c c c c c c c c}0&0&1&1&1&0&0\\&&1\end{array}}+
e_{\tiny\begin{array}{c c c c c c c c}0&1&1&1&1&0&0\\&&0\end{array}}+
e_{\tiny\begin{array}{c c c c c c c c}0&1&2&1&1&1&0\\&&1\end{array}}$ \\
& $p=2$ & $(D_7)_2$ & \def\arraystretch{0.5} \arraycolsep=0pt $e_{\tiny\begin{array}{c c c c c c c c}1&0&0&0&0&0&0\\&&0\end{array}}+
e_{\tiny\begin{array}{c c c c c c c c}0&1&1&0&0&0&0\\&&1\end{array}}+
e_{\tiny\begin{array}{c c c c c c c c}0&0&1&1&0&0&0\\&&1\end{array}}+
e_{\tiny\begin{array}{c c c c c c c c}0&1&1&1&0&0&0\\&&0\end{array}}+
e_{\tiny\begin{array}{c c c c c c c c}0&0&1&1&1&0&0\\&&0\end{array}}+
e_{\tiny\begin{array}{c c c c c c c c}0&0&0&1&1&1&0\\&&0\end{array}}+
e_{\tiny\begin{array}{c c c c c c c c}0&0&0&0&1&1&1\\&&0\end{array}}+
e_{\tiny\begin{array}{c c c c c c c c}1&1&1&1&1&1&1\\&&1\end{array}}$ \\
& $p=2$ & $(D_5 A_2)_2$ & \def\arraystretch{0.5} \arraycolsep=0pt $e_{\tiny\begin{array}{c c c c c c c c}0&0&0&0&0&1&1\\&&0\end{array}}+
e_{\tiny\begin{array}{c c c c c c c c}0&0&0&0&1&1&1\\&&0\end{array}}+
e_{\tiny\begin{array}{c c c c c c c c}1&1&1&1&0&0&0\\&&1\end{array}}+
e_{\tiny\begin{array}{c c c c c c c c}1&1&1&1&1&0&0\\&&0\end{array}}+
e_{\tiny\begin{array}{c c c c c c c c}0&1&2&1&0&0&0\\&&1\end{array}}+
e_{\tiny\begin{array}{c c c c c c c c}0&1&1&1&1&0&0\\&&1\end{array}}+
e_{\tiny\begin{array}{c c c c c c c c}0&0&1&1&1&1&0\\&&1\end{array}}+
e_{\tiny\begin{array}{c c c c c c c c}0&1&1&1&1&1&0\\&&0\end{array}}$ \\
& $p=3$ & $(A_7)_3$ & \def\arraystretch{0.5} \arraycolsep=0pt $e_{\tiny\begin{array}{c c c c c c c c}0&0&0&1&1&1&0\\&&0\end{array}}+
e_{\tiny\begin{array}{c c c c c c c c}0&0&0&0&1&1&1\\&&0\end{array}}+
e_{\tiny\begin{array}{c c c c c c c c}1&1&1&0&0&0&0\\&&1\end{array}}+
e_{\tiny\begin{array}{c c c c c c c c}1&1&1&1&0&0&0\\&&0\end{array}}+
e_{\tiny\begin{array}{c c c c c c c c}0&0&1&1&1&0&0\\&&1\end{array}}+
e_{\tiny\begin{array}{c c c c c c c c}0&1&1&1&1&0&0\\&&0\end{array}}+
e_{\tiny\begin{array}{c c c c c c c c}0&1&2&1&0&0&0\\&&1\end{array}}+
e_{\tiny\begin{array}{c c c c c c c c}0&0&1&1&1&1&1\\&&0\end{array}}$\\
\hline
\caption{Eminent nilpotent representatives in exceptional types}\label{t:exnil}\end{longtable}\end{center}
\def\arraystretch{\oldarraystretch} \arraycolsep=\oldarraycolsep

\begin{center}\begin{longtable}{|c|c|c|c|}\hline Type & Char. & Representation & Representative\\\hline
$A_n$ & any & $V(n+1)$ & $x_{\alpha_1}(1) \cdots x_{\alpha_n}(1)$\\
\hline

$B_n$; $n\geq 2$ & $p \neq 2$ & $V(2n+1)$ & $x_{\alpha_1}(1) \cdots x_{\alpha_n}(1)$\\
 & $p = 2$ & $V(2n) + R$ & $x_{\alpha_1}(1) \cdots x_{\alpha_n}(1)$\\

\hline
$C_n$; $n\geq 2$ & any & $V(2n)$ & $x_{\alpha_1}(1) \cdots x_{\alpha_n}(1)$\\
\hline
$D_n$; $n\geq 4$ & $p=2$ & $V(2n-2l+2) + V(2l-2)$ & $\left(\prod_{i=1}^{n-1}x_{\alpha_i}(1)\right) x_{\beta_l}(1)$, where $\beta_l = \alpha_{n} + \sum_{i=2l-n-1}^{n-2} \alpha_{i}$\\
& & $\frac{n+1}{2}<l < n$ & \\
\hline
\caption{Eminent unipotent representatives in classical types}\label{t:classunip}\end{longtable}\end{center}

\begin{center}\begin{table}[h!]\begin{tabular}[t]{|c|c|c|}\hline Type & Char. & Label \\\hline
$G_2$ & any & $G_2$  \\
& $p=3$ & $(\tilde{A}_1)_3$  \\ \hline 
$F_4$ & any & $F_4$ \\
& $p=2$ & $F_4(a_2)$  \\
& $p=2$ & $(\tilde{A}_2 A_1)_2$ \\
\hline
$E_6$ & any & $E_6$ \\
& any & $E_6(a_1)$  \\
& $p=2$ & $E_6(a_3)$  \\
\hline
$E_7$ & any & $E_7$  \\
& any & $E_7(a_1)$  \\
& any & $E_7(a_2)$  \\
& $p=2$ & $E_7(a_3)$ \\\hline \end{tabular}
\begin{tabular}[t]{|c|c|c|}\hline Type & Char. & Label \\\hline
$E_8$ & any & $E_8$ \\
& any & $E_8(a_1)$  \\
& any & $E_8(a_2)$  \\
& $p=2$ & $E_8(a_3)$  \\
& $p=2$ & $E_8(a_4)$  \\
& $p=2$ & $E_8(a_5)$  \\
& $p=2$ & $(D_5 A_2)_2$ \\
& $p=3$ & $(A_7)_3$ \\
\hline\end{tabular}
\caption{Eminent unipotent representatives in exceptional types}\label{t:exunip}
\end{table}\end{center}
It is often useful to have a means of recognising unipotent and nilpotent elements in exceptional types from easily computable invariants. For unipotent classes, this was accomplished by Lawther {\cite[Section 3]{L09}}. Let $C = C_{\Lie(G)}(x)$ for $x$ a unipotent or nilpotent element. We use the following abbreviations.

\hspace{1cm} JBS := List of Jordan block sizes

\hspace{1cm} DS := Dimensions of the terms in the derived series of $C$ 

\hspace{1cm} LS := Dimensions of the terms in the lower central series of $C$

\hspace{1cm} ALG := Dimension of the Lie algebra of derivations of $C$

\hspace{1cm} ALG' := Dimension of the derived subalgebra of the Lie algebra of derivations of $C$

\hspace{1cm} NIL := Dimension of the Nilradical of $C$

\hspace{1cm} NDS := Dimensions of the Lie normaliser of each term in the derived series of $C$

\begin{theorem}\label{p:nilprec} Let $G$ be an exceptional $k$-group. Then Table \ref{t:nilprec} gives a list of data which is sufficient to identify the class of any nilpotent or unipotent element of $G$.
\end{theorem}

\begin{center}\begin{table}\begin{tabular}[t]{|c|c|l|}
\hline
Type & Char. & Data for nilpotent orbits \\\hline
$G_2$ & $p \neq 3$ & JBS on $\Vm$\\
& $p=3$ & JBS on $\Lie(G)$\\
$F_4$ & $p=2$ & JBS on $\Vm$ and $\Lie(G)$; DS \\
& $p=3$ & JBS on $\Vm$ and $\Lie(G)$ \\
& $p \geq 5$ & JBS on $\Vm$ \\
$E_6$ & $p =2$ & JBS on $\Lie(G)$; DS \\
& $p=3$ & JBS on $\Lie(G)$; ALG \\
& $p \geq 5$ & JBS on $\Vm$\\
$E_7$ & $p=2$ & JBS on $\Lie(G)$; DS; ALG \\
& $p=3$ & JBS on $\Vm$ and $\Lie(G)$; DS \\
& $p \geq 5$ & JBS on $\Vm$\\
$E_8$ & $p=2$ & JBS on $\Lie(G)$; DS; ALG; ALG' \\
& $p=3$ & JBS on $\Lie(G)$; NDS; NIL \\ 
& $p=5$ & JBS on $\Lie(G)$; DS \\
& $p \geq 5$ & JBS on $\Lie(G)$ \\
\hline\end{tabular}
\begin{tabular}[t]{|c|c|l|}
\hline
Type & Char. & Data for unipotent classes \\\hline
$G_2$ & $p \neq 3$ & JBS on $\Vm$\\
& $p = 3$ & JBS on $\Lie(G)$ \\
$F_4$ & $p \leq 3$ & JBS on $\Vm$ and $\Lie(G)$\\
& $p \geq 5$ & JBS on $\Vm$\\
$E_6$ & $p \neq 3$ & JBS on $\Vm$\\
& $p = 3$ & JBS on $\Lie(G)$ \\
 $E_7$ & $p=2$ & JBS on $\Lie(G)$; LS \\
 & $p = 3$ & JBS on $\Vm$ and $\Lie(G)$ \\
 & $p \geq 5$ & JBS on $\Vm$\\
$E_8$ & $p=2$ & JBS on $\Lie(G)$; LS \\
& $p \geq 3$ & JBS on $\Lie(G)$ \\
 \hline\end{tabular}
\caption{Recognition data for nilpotent orbits and unipotent classes.}\label{t:nilprec}\end{table}\end{center}

\begin{remarks}(i) One important application of Theorem \ref{p:nilprec} is its use in the proof of Theorem \ref{nilreps}(i) for $G$ of exceptional type. To find the eminent nilpotent elements in $\Lie(G)$  the strategy is to list all maximal subsystem subalgebras up to conjugacy; for each $\m$ in that list, construct representatives of all distinguished orbits; and then use Theorem \ref{p:nilprec} to identify their class in $G$.

(ii) We supply auxiliary tables in Section \ref{s:auxtabs}; whenever at least two nilpotent orbits have the same Jordan block structure on $\Vm$ (when it exists) and $\text{Lie}(G)$ we list them in a table along with the data required to distinguish them, using the same abbreviations given before Proposition \ref{p:nilprec}.  
\end{remarks}

\section{Preliminaries} \label{s:prelims}
\subsection{Notation}\label{s:notation} Throughout $G$ is a reductive $k$-group, and most often $G$ is simple. Fix a Borel subgroup $B\leq G$ containing a maximal torus $T$. This defines a root system $\Phi$ with base $\Delta=\{\alpha_1,\dots,\alpha_l\}$ of simple roots generating the positive roots $\Phi^+$ as positive integral sums. We use the Bourbaki ordering for a Dynkin diagram a corresponding roots and follow \cite{Bourb82} in denoting roots by a string of numbers in the form of a Dynkin diagram corresponding to coefficients of a root expressed in terms of the simple roots. For example, the highest short root of an $F_4$ system is denoted $1232$.

The adjoint action of $G$ on $\Lie(G)$ defines a Cartan decomposition $\Lie(G)\cong \Lie(T)\oplus\bigoplus_{\alpha\in\Phi}ke_\alpha$. Corresponding to the root spaces $ke_\alpha$ one has the one-parameter subgroups $U_\alpha<G$, where each $U_\alpha\cong \mathbb{G}_a$ is a root groups.   In case $G$ is simple and simply connected, we may insist that the $e_\alpha$ form a Chevalley basis; usually, this is only recognised as canonical up to a choice of the sign of a coefficient $\lambda$ for certain triples of roots $(\alpha,\beta,\gamma)$ such that $[e_\alpha,e_\beta]=\lambda e_\gamma$. These choices determine the structure constants for $\Lie(G)$ \cite[Proposition 4.2.2]{Car89}. Through exponentiation one can fix an isomorphism $ke_\alpha\to U_\alpha$ in which $ce_\alpha\mapsto x_\alpha(c)$ for all $c\in k$.

The computations done in Section \ref{s:proofrecs} were performed with GAP4 \cite{GAP} in the simply connected case, using the representatives for the nilpotent orbits given in \cite{Ste16} (see the latest arXiv version for a correction in type $F_4$). In the adjoint case we used \textsc{Magma} \cite{Magma}. Although \textsc{Magma} uses a different choice of signs in the structure constants compared to GAP4, the same representatives can still be used and this is justified as follows.

All representatives listed in \cite{Ste16} are of the form $e = \sum_{i = 1}^t e_{\beta_i}$, where $\beta_1, \ldots, \beta_t \in \Phi^+$ are $\Z$-independent roots. Thus for any choice of signs $c_1, \ldots, c_t \in \{1,-1\}$ there exists a semisimple element in $G$ that conjugates $e$ to the nilpotent element $\sum_{i = 1}^t c_i e_{\beta_i}$; see for example \cite[Lemma 16.2C]{HumphreysGroupBook}. In particular, the representatives in \cite{Ste16} do not depend on the signs of structure constants in the Chevalley basis of $\Lie(G)$.

\subsection{Eminent elements and generalised subsystem subgroups and subalgebras} 
Recall that a subgroup $M$ of $G$ is a \emph{subsystem subgroup} if it is semisimple (the definition of which we assume to include connected) and normalised by a maximal torus.\footnote{Unless $G$ has a factor with root system $\Phi$ such that the pair $(\Phi,p)$ is $(B_n,2)$, $(C_n,2)$, $(F_4,2)$ or $(G_2,3)$, this implies it has a symmetric root system which is closed under sums. In any case, the relevant root systems can be obtained from the Dynkin diagram using a version of the Borel--de Siebenthal algorithm; see \cite[Prop.~13.15]{MT}.} Note that the derived subgroup of any Levi subgroup is a subsystem subgroup.  Define also a \emph{subsystem subalgebra} to be a subalgebra $\m\leq \Lie(G)$ with $\m=\Lie(M)$ for some subsystem subgroup $M$. 

Define a \emph{generalised subsystem subgroup} as follows. If $M$ is a subsystem subgroup of $G$ containing some simple factors $M_1,\dots,M_t$ of type $D_{r_1},\dots,D_{r_t}$ where $t\geq 0$, then there is a semisimple subgroup $M_{i,l}$ of $M_i$ of type $B_{l}B_{r_i-l-1}$ with $0\leq l\leq r_i-1$ (see the proof of Lemma \ref{l:maxgensubsys}). Then the generalised subsystem subgroups are obtained by replacing some subset of the $M_i$ with the subgroups $M_{i,l}$. Accordingly, we define a \emph{generalised subsystem subalgebra} $\m$ to be $\Lie(M)$ for a generalised subsystem subgroup $M$.

Recall that a unipotent or nilpotent element $x$ is \emph{distinguished} if it is contained in no proper Levi of $G$. The following definition is more restrictive.

\begin{defn}
A unipotent element $u \in G$ is called {\em eminent} if it is not contained in any proper generalised subsystem subgroup of $G$. A nilpotent element $e \in \Lie(G)$ is called {\em eminent} if it is not contained in any proper generalised subsystem subalgebra of $\Lie(G)$. 
\end{defn}

For use in the proof of Theorem \ref{nilreps} we record the conjugacy classes of maximal generalised subsystem subgroups which are not Levi subgroups. The maximal generalised subsystem subalgebras follow immediately.  

\begin{lemma} \label{l:maxgensubsys}
Suppose that $G$ is a classical group. Let $M$ be a generalised subsystem subgroup which is maximal amongst proper ones, such that $M$ is not a Levi subgroup. Then $M$ is conjugate to precisely one subgroup with root system as given in the following table. Moreover, every subgroup in the table is a maximal generalised subsystem subgroup.  
\end{lemma}
\pagebreak
\begin{center}\begin{longtable}{|c|c|c|c|}\hline $G$ & Char. & $M$ \\\hline
$A_n$ & & \\
\hline
$B_n$; $n\geq 2$ & $p \neq 2$ & $B_m D_{n-m}$; $1 \leq m \leq n$ \\
& $p=2$ & $B_m B_{n-m}$; $1 \leq m \leq \frac{n}{2}$ \\
& any & $D_n$ \\
\hline
$C_n$; $n\geq 2$ & any & $C_m C_{n-m}$; $1 \leq m \leq \frac{n}{2}$ \\
& $p = 2$ & $D_n$ \\
\hline
$D_n$; $n\geq 4$ & any & $D_m D_{n-m}$; $2 \leq m \leq \frac{n}{2}$ \\
& any & $B_{n-1}$ \\
 & $p \neq 2$ & $B_m B_{n-m-1}$; $1 \leq m \leq \frac{n}{2}$ \\
 \hline
\caption{Maximal generalised subsystem subgroups}\label{t:maxgensubsys}\end{longtable}\end{center}

\begin{proof}
For the subsystem subgroups this is a routine use of the Borel--de Siebenthal algorithm for closed subsystems \cite[Prop.~13.12]{MT} and \cite[Prop.~13.15(i),(ii)]{MT} for the $2$-closed non-closed subsystems. When $G$ has type $D_n$ further consideration is required. The definition of generalised subsystem subgroups immediately implies that they will not be maximal unless we replace $D_n$ itself with $B_{m}B_{n-m-1}$ with $0\leq m \leq n-1$. When $p=2$, the subgroups of type $B_{m}B_{n-m-1}$ with $1\leq m \leq n-1$ are contained in $B_{n-1}$ (indeed they are subsystem subgroups of $B_{n-1}$ since $p=2$) and so they are not maximal. We need to check that these subgroups are unique up to conjugacy and maximal amongst proper generalised subsystem subgroups. When $p \neq 2$, the subgroups of type $B_{m}B_{n-m-1}$ are the connected stabiliser of a decomposition of $V$ into non-degenerate subspaces of dimensions $2m+1$ and $2n-2m-1$ and thus are unique up to conjugacy since $G$ acts transitively on non-degenerate subspaces of dimension $k$. Moreover, this implies they are maximal amongst all connected proper subgroups of $G$ and thus amongst all proper generalised subsystem subgroups. When $p=2$, a subgroup of type $B_{n-1}$ is the stabiliser of a non-singular vector and is again unique up to conjugacy and maximal amongst all connected proper subgroups of $G$. 
\end{proof}

We observe that eminence is invariant under isogenies.

\begin{lemma} \label{l:isogeny}
Let $\phi:G\to H$ be an isogeny of reductive $k$-groups. 
\begin{enumerate}[\normalfont (i)]
\item The map $\phi$ induces a bijection between the unipotent varieties of $G$ and $H$, and the set of generalised subsystem subgroups of $G$ and $H$. 
\item If $\phi$ is a central isogeny then the map $d\phi$ induces a bijection between the nilpotent cones of $G$ and $H$, and the set of generalised subsystem subalgebras of $G$ and $H$.
\end{enumerate}
\end{lemma}
\begin{proof}
The bijections of unipotent varieties and nilpotent cones follow from \cite[Proposition~5.1.1]{Car93} and \cite[2.7]{Jan04}. Since the kernel of $\phi$ is central, it is contained in any maximal torus; as $\phi$ is surjective with finite kernel it follows that $\phi$ is a bijection on maximal tori. Thus $M$ is a subsystem subgroup if and only if $\phi(M)$ is. If $M$ is a generalised subsystem subgroup, the result follows from the definition.
\end{proof}

\subsection{Unipotent classes and nilpotent orbits in classical groups} \label{s:classical}
Let $G$ be simple of classical type. In light of Lemma \ref{l:isogeny}, for the purposes of this paper it does no harm for us to assume that $G = \SL(V)$, $\Sp(V)$, or $\SO(V)$. 

We will choose forms preserved by $G$ on $V$ consistent with \cite{Jan73}. Let $V$ have $k$-basis $\{v_1, v_2, \ldots, v_m\}$. 
For $G=\Sp(V)$ and $m = 2n$, we choose a non-degenerate alternating bilinear form on $V$ defined by $(v_i, v_{n+i}) = 1 = -(v_{n+i}, v_i)$ and $(v_i, v_j) = 0$ for all other $i,j$. For $G=\SO(V)$ with $V$ of odd dimension $m = 2n+1$, choose a non-degenerate symmetric bilinear form on $V$ defined by $(v_i, v_{n+i}) = 1 = (v_{n+i}, v_i)$, $(v_{2n+1},v_{2n+1}) = 2$ and $(v_i, v_j) = 0$ for all other $i, j$. When $p=2$, we also define a quadratic form on $V$ by $Q(v_i) = 0$ for $i \leq 2n$, $Q(v_{2n+1})=1$, such that $Q$ has polarization $(-,-)$ on $V$. Finally, for $\SO(V)$ with $V$ of even dimension $m = 2n$ choose a non-degenerate symmetric bilinear form on $V$ defined by $(v_i, v_{n+i}) = 1 = (v_{n+i}, v_i)$ and $(v_i, v_j) = 0$ for all other $i, j$. And in this case for $p=2$, define a quadratic form on $V$ by $Q(v_i) = 0$ for all $i$, with polarization $(-,-)$ on $V$.

If $G = \SL(V)$, the class of a unipotent or nilpotent element $x \in G$ is determined by its Jordan block sizes on $V$. In what follows, we will describe the unipotent classes and nilpotent orbits for $G = \Sp(V)$ and $G = \SO(V)$. In this case for $x$ a unipotent element or a nilpotent element, we have a decomposition $V \downarrow k[x] = V_1 \perp \cdots \perp V_t$, where the $V_i$ are \emph{orthogonally indecomposable} $k[x]$-modules. Here orthogonally indecomposable means that if $V_i = W \perp W'$ as $k[x]$-modules, then $W = 0$ or $W' = 0$. Hesselink \cite[3.5]{Hesselink1979} determined the orthogonally indecomposable $k[x]$-modules that can occur, and thus classifed the unipotent classes and nilpotent orbits for $\Sp(V)$ and $\mathrm{O}(V)$. The possibilities for orthogonally indecomposable $k[x]$-modules are as follows:

\begin{itemize}
\item The $k[x]$-module $V(m)$ is non-degenerate of dimension $m$ and $x$ acts on $V(m)$ with a single Jordan block of size $m$. 
\item The $k[x]$-module $W(m)$ has dimension $2m$, with $W(m) = W_1 \oplus W_2$ where $W_i$ are totally singular of dimension $m$ and $x$ acts on each $W_i$ with a single Jordan block of size $m$. 
\item The $k[x]$-module $W_l(m)$ only occurs when $p=2$ and $x$ is a nilpotent element. It is non-degenerate of dimension $2m$ and $x$ acts with two Jordan blocks of size $m$. If $G = \Sp(V)$ then $0 < l < m/2$ and $l = \text{min}\{n : (x^{n+1}(v),x^{n}(v)) = 0 \text{ for all } v \in W_l(m)\}$. If $G = \SO(V)$ then $(m+1)/2 < l \leq m$ and $l = \text{min}\{n : Q(x^{n}(v)) = 0 \text{ for all } v \in W_l(m)\}$. 
\item The $k[x]$-module $D(m)$ only occurs when $G = \SO(V)$ with $\dim(V)$ odd, $p=2$ and $x$ is a nilpotent element. It is degenerate of dimension $2m - 1$ and $x$ acts with two Jordan blocks of sizes $m$ and $m-1$. 
\item The $k[x]$-module $R$ only occurs when $G = \SO(V)$ with $\dim(V)$ odd, $p=2$ and $x$ is a unipotent element. It is the $1$-dimensional radical of $V$ and so $x$ acts trivially on it. 
\end{itemize}

It is clear that the decomposition of $V \downarrow k[x]$ into orthogonally indecomposable summands determines the class of $x$. There are various ways of writing down the decomposition $V \downarrow k[x]$, we will use the \emph{distinguished normal form} introduced in \cite{LS12}. Since we are interested in the simple group $\SO(V)$, we state when a class in $\mathrm{O}(V)$ meets the subgroup $\SO(V)$ and if so, whether it splits into two $\SO(V)$-classes; see \cite[Lemma~3.11, Propositions~5.25, 6.22]{LS12}. The following lemma is a combination of results from \cite[Chs.~5.2, 5.3, 5.6, 6.2, 6.8]{LS12}.

\begin{lemma} \label{l:decomps}
Let $G = \Sp(V)$ or $\mathrm{O}(V)$ and let $x$ be unipotent or nilpotent element of $G$ or $\Lie(G)$, respectively. The following list states which orthogonal decompositions $V\downarrow k[x]$ can occur. Conversely, the multiset of orthogonal factors determines $x$ uniquely up to conjugacy. 

Furthermore, $x$ is distinguished if and only if $r=0$.
\begin{enumerate}[\normalfont (i)]
\item If $p \neq 2$, then
$$V \downarrow k[x] = \sum_{i=1}^r W(m_i)^{a_i} \oplus \sum_{j=1}^s V(n_j),$$ where the $n_j$ are distinct, and all are even (resp. odd) when $G = \Sp(V)$ (resp. $\mathrm{O}(V)$). For $G = \mathrm{O}(V)$ and $x$ unipotent, the class meets $\SO(V)$. For $G = \mathrm{O}(V)$, a class splits into two $\SO(V)$-classes if and only if $s=0$ and $a_i = 0$ for all odd $m_i$. 

\item If $x$ is unipotent, $p =2$ and $V$ is of even dimension, then 

$$V \downarrow k[x] = \sum_{i=1}^r W(m_i)^{a_i} \oplus \sum_{j=1}^t V(2n_j)^{b_j}$$ where the $n_j$ are distinct and $b_j \leq 2$ for all $j$. Such a class meets $\SO(V)$ if $\sum_j b_j$ is even and it splits if and only if $t=0$. 

\item If $x$ is unipotent, $p =2$ and $V$ is of odd dimension, then 
$$V \downarrow k[x] = \sum_{i=1}^r W(m_i)^{a_i} \oplus \sum_{j=1}^t V(2n_j)^{b_j} \oplus R$$ where the $n_j$ are distinct and $b_j \leq 2$.

\item If $x$ is nilpotent, $p =2$ and $G = \Sp(V)$, then 

 $$V \downarrow k[x] = \sum_{i=1}^r W(m_i)^{a_i} \oplus \sum_{j=1}^s W_{l_j}(n_j) \oplus \sum_{k=1}^{t} V(2q_k)^{b_k}$$ where the sequences $(n_j)$, $(l_j)$ and $(n_j-l_j)$ are strictly decreasing, $b_k \leq 2$ and for all $j,k$ either $q_k > n_j - l_j$ or $q_k < l_j$.

\item If $x$ is nilpotent, $p =2$ and $G = \mathrm{O}(V)$ with $V$ of even dimension, then  

$$V \downarrow k[x] = \sum_{i=1}^r W(m_i)^{a_i} \oplus \sum_{j=1}^s W_{l_j}(n_j)$$ where the sequences $(n_j)$, $(l_j)$ and $(n_j-l_j)$ are strictly decreasing. Each class splits into two $\SO(V)$-classes if and only if $s=0$.

\item If $x$ is nilpotent, $p=2$ and $V$ is of odd dimension, then

$$V \downarrow k[x] = \sum_{i=1}^r W(m_i)^{a_i} \oplus \sum_{j=1}^s W_{l_j}(n_j) \oplus D(m)$$ where the sequences $(n_j)$, $(l_j)$ and $(n_j-l_j)$ are strictly decreasing and $m < l_s$. 

\end{enumerate}

\end{lemma}

\begin{remark}\label{regularelts}Each simple $k$-group of classical type has a unique class of nilpotent and unipotent elements of highest dimension called \emph{regular}. A representative of the regular class is well-known to be $\prod_i x_{\alpha_i}(1)$ for unipotent elements \cite[3.2]{Steinberg65} and $\sum_i e_{\alpha_i}$ for nilpotent elements by \cite[5.9]{Springer66} and \cite{Keny87}, where the product and sum are supported on all simple roots. From \cite[3.3.6 and pp. 60--61]{LS12} we see that the regular elements are usually the only ones which have a Jordan block of largest possible size. In type $A$, this means that $x$ acts with a single Jordan block. For the other classical cases from the lemma, with $m = \dim(V)$:

\begin{enumerate}\item $V \downarrow k[x] = V(m-1) + V(1)$ if $G = \SO(V)$ and $m$ is even, otherwise $V \downarrow k[x] = V(m)$;
\item $V \downarrow k[x] = V(m)$ if $G = \Sp(V)$, and $V \downarrow k[x] = V(m-2) + V(2)$ if $G = \SO(V)$;
\item $V \downarrow k[x] = V(m-1) + R$;
\item $V \downarrow k[x] = V(m)$;
\item $V \downarrow k[x] = W_{m/2}(m/2)$;
\item $V \downarrow k[x] = D((m+1)/2)$. \end{enumerate}
\end{remark}

\subsection{Recovering all classes from the eminent ones} \label{s:allclasses}

We present a brief discussion of how one can find a full set of representatives of the unipotent classes and nilpotent orbits iteratively from a set of representatives of the eminent classes and orbits. (If $G$ is of exceptional type, one can simply look at the full list in \cite{Ste16}.) Since an element $x$ is eminent in some generalised subsystem it suffices to give a set of Chevalley generators for maximal generalised subsystem subgroups. The Borel-de Siebenthal algorithm solves the problem in the case of subsystems. Hence the following lemma completes the picture.

\begin{lemma}\label{lemma:typeDoddsystem}
Let $G$ be a simple $k$-group of type $D_n$ with $p\neq 2$. Let $\beta = \alpha_m + \alpha_{m+1} + \ldots + \alpha_{n-2}$ and define: 
$$ H_1 = \langle U_{\pm\alpha_1}, \ldots, U_{\pm\alpha_{m-1}}, x_{\beta+\alpha_{n-1}}(t)x_{\beta+\alpha_{n}}(-t), x_{-\beta-\alpha_{n-1}}(t)x_{-\beta-\alpha_{n}}(-t) \mid t \in k \rangle,$$ 
$$ H_2 = \langle U_{\pm\alpha_{m+1}}, \ldots, U_{\pm\alpha_{n-2}}, x_{\alpha_{n-1}}(t)x_{\alpha_{n}}(t), x_{-\alpha_{n-1}}(t)x_{-\alpha_{n}}(t) \mid t \in k \rangle.$$
Then the subgroups $H_1$ and $H_2$ commute, $H_1$ has type $B_m$ and $H_2$ has type $B_{n-m-1}$. 
\end{lemma}

\begin{proof}
This is proved in \cite[pp.67--68]{Tes95}, with the generators amended to reflect our choice of stucture constants.
\end{proof}

If $p\neq 2$, then Theorem \ref{nilreps} implies that the unique eminent class of $x$ in type $B_m$ is regular. Then in view of the lemma, an eminent unipotent element in the subsystem subgroup $H:=H_1H_2$ can be represented by  $u = u_1u_2$, where \[u_1 = x_{\alpha_1}(1) \cdots x_{\alpha_{m-1}}(1) x_{\beta+\alpha_{n-1}}(1)x_{\beta+\alpha_{n}}(-1)\text{\quad and\quad}u_2 = x_{\alpha_{m+1}}(1) \cdots x_{\alpha_{n-2}}(1) x_{\alpha_{n-1}}(1)x_{\alpha_{n}}(1).\] Similarly, if $e$ is a regular nilpotent element of $\Lie(H)$, then $e$ is conjugate to $e_{\beta+\alpha_{n-1}}-e_{\beta+\alpha_{n}} + \sum_{i=1,i\neq m}^{n}( e_{\alpha_i})$. For even more representatives that work regardless of the choice of signs of the structure constants, see Remark \ref{remark:Duniprep}.

\subsection{Representatives for classical groups in small dimensions}

To illustrate our main theorem, Tables \ref{t:sp4reps}--\ref{t:so10reps} provide representatives for the unipotent classes and nilpotent orbits in classical algebraic groups of small dimension, as considered in \cite[Section 8]{LS12}.

Type $A_n$ is easily described. The class of $x$ corresponds to a partition $d_1 + \cdots + d_t = n+1$. Viewed as an element of $\SL_{n+1}$ and put in Jordan normal form, the elements of the superdiagonal that $x$ is supported on correspond to the simple roots that it is supported on. More specifically, a unipotent representative is $\prod_{k = 1}^n x_{\alpha_k}(\varepsilon_k)$, and a nilpotent representative is $\sum_{k = 1}^n \varepsilon_ke_{\alpha_k}$, where $\varepsilon_k = 1$ if $\sum_{j = 1}^{i-1} d_j < k < \sum_{j = 1}^i d_j$ for some $1 \leq i \leq t$, and $\varepsilon_k = 0$ otherwise. 

In the remaining explicit examples we give, there is an injective map from the set of unipotent classes to the set of nilpotent orbits and we have chosen representatives in such a way that a representative for a unipotent class is given by $\prod_{i = 1}^t x_{\beta_i}(1)$, where $\sum_{i = 1}^t e_{\beta_i}$ is the corresponding representative for the corresponding nilpotent orbit. Therefore we only list representatives for nilpotent orbits, apart from $\text{SO}_7(k)$ where the injection is not as natural. When $p \neq 2$, the characteristic is good for $G$, and so there is a Springer map, hence \emph{a fortiori} a bijection from the set of unipotent classes to the set of nilpotent orbits. When $p = 2$ and $G=\text{Sp}_{2n}$, a unipotent class is mapped to a nilpotent orbit such that the decompositions in Lemma \ref{l:decomps} (ii) and (iv) coincide. For a unipotent element of $\text{SO}_{2n}(k)$, the decomposition of the corresponding nilpotent element is described in \cite[6.3]{LS12}. In the tables for the orthogonal groups, some representatives have two rows of decompositions. The first row corresponds to the case where $p \neq 2$, and the second row corresponds to the case where $p=2$. 

\pagebreak
\begin{center}\begin{longtable}{|l|l|l|}
\hline
Representative $x$ & Char. & $V \downarrow k[x]$ \\
\hline 
$0$ & any & $W(1)^2$   \\
\def\arraystretch{0.5}\arraycolsep=0pt$e_{\tiny\begin{array}{c c}1&0\end{array}}$ & any & $W(2)$ \\
\def\arraystretch{0.5}\arraycolsep=0pt$e_{\tiny\begin{array}{c c}0&1\end{array}}$ & any & $W(1) + V(2)$ \\
\def\arraystretch{0.5}\arraycolsep=0pt$e_{\tiny\begin{array}{c c}0&1\end{array}} + e_{\tiny\begin{array}{c c}2&1\end{array}}$ & $p=2$ & $V(2)^2$ \\ 
\def\arraystretch{0.5}\arraycolsep=0pt$e_{\tiny\begin{array}{c c}1&0\end{array}} + e_{\tiny\begin{array}{c c}0&1\end{array}}$ & any & $V(4)$ \\
 \hline
 \caption{Representatives for nilpotent orbits and unipotent classes of $\text{Sp}_4(k)$. \label{t:sp4reps}}
\end{longtable}\end{center}


\begin{center}\begin{longtable}{|l|l|l|}
\hline
 Representative $x$ & Char. & $V \downarrow k[x]$ \\
\hline 
$0$ & any & $W(1)^3$ \\
\def\arraystretch{0.5}\arraycolsep=0pt$e_{\tiny\begin{array}{c c c}1&0&0\end{array}} + e_{\tiny\begin{array}{c c c}0&1&0\end{array}}$ & any & $W(3)$  \\
\def\arraystretch{0.5}\arraycolsep=0pt$e_{\tiny\begin{array}{c c c}1&0&0\end{array}}$ & any & $W(2) + W(1)$  \\
\def\arraystretch{0.5}\arraycolsep=0pt$e_{\tiny\begin{array}{c c c}1&0&0\end{array}} + e_{\tiny\begin{array}{c c c}0&0&1\end{array}}$ & any & $W(2) + V(2)$ \\
\def\arraystretch{0.5}\arraycolsep=0pt$e_{\tiny\begin{array}{c c c}0&1&0\end{array}} + e_{\tiny\begin{array}{c c c}0&0&1\end{array}}$ & any & $W(1) + V(4)$ \\
\def\arraystretch{0.5}\arraycolsep=0pt$e_{\tiny\begin{array}{c c c}0&0&1\end{array}} + e_{\tiny\begin{array}{c c c}2&2&1\end{array}}$ & $p=2$ & $W(1) + V(2)^2$ \\
\def\arraystretch{0.5}\arraycolsep=0pt$e_{\tiny\begin{array}{c c c}0&0&1\end{array}}$ & any & $W(1)^2 + V(2)$ \\
\def\arraystretch{0.5}\arraycolsep=0pt$e_{\tiny\begin{array}{c c c}0&1&0\end{array}} + e_{\tiny\begin{array}{c c c}0&0&1\end{array}} + e_{\tiny\begin{array}{c c c}2&2&1\end{array}}$ & any & $V(4) + V(2)$ \\
\def\arraystretch{0.5}\arraycolsep=0pt$e_{\tiny\begin{array}{c c c}1&0&0\end{array}} + e_{\tiny\begin{array}{c c c}0&1&0\end{array}} + e_{\tiny\begin{array}{c c c}0&0&1\end{array}}$ & any & $V(6)$  \\
\hline
\def\arraystretch{0.5}\arraycolsep=0pt$e_{\tiny\begin{array}{c c c}1&0&0\end{array}} + e_{\tiny\begin{array}{c c c}0&1&0\end{array}} + e_{\tiny\begin{array}{c c c}2&2&1\end{array}}$ (nilp. only) & $p=2$ & $W_1(3)$  \\ 
\hline
\caption{Representatives for nilpotent orbits and unipotent classes of $\text{Sp}_6(k)$. \label{t:sp6reps}}
\end{longtable}\end{center}


\begin{center}\begin{longtable}{|l|l|l|}
\hline
Representative $x$ & Char. & $V \downarrow k[x]$ \\
\hline 
 $0$ & any & $W(1)^4$ \\
\def\arraystretch{0.5}\arraycolsep=0pt$e_{\tiny\begin{array}{c c c c}1&0&0&0\end{array}} + e_{\tiny\begin{array}{c c c c}0&1&0&0\end{array}} + e_{\tiny\begin{array}{c c c c}0&0&1&0\end{array}}$
 & any & $W(4)$  \\

\def\arraystretch{0.5}\arraycolsep=0pt$e_{\tiny\begin{array}{c c c c}1&0&0&0\end{array}} + e_{\tiny\begin{array}{c c c c}0&1&0&0\end{array}}$ & any & $W(3) + W(1)$ \\ 

\def\arraystretch{0.5}\arraycolsep=0pt$e_{\tiny\begin{array}{c c c c}1&0&0&0\end{array}} + e_{\tiny\begin{array}{c c c c}0&1&0&0\end{array}} + e_{\tiny\begin{array}{c c c c}0&0&0&1\end{array}}$ & any & $W(3) + V(2)$ \\ 

\def\arraystretch{0.5}\arraycolsep=0pt$e_{\tiny\begin{array}{c c c c}1&0&0&0\end{array}} + e_{\tiny\begin{array}{c c c c}0&0&1&0\end{array}}$ & any & $W(2)^2$ \\

\def\arraystretch{0.5}\arraycolsep=0pt$e_{\tiny\begin{array}{c c c c}1&0&0&0\end{array}}$ & any & $W(2) + W(1)^2$ \\

\def\arraystretch{0.5}\arraycolsep=0pt$e_{\tiny\begin{array}{c c c c}1&0&0&0\end{array}} + e_{\tiny\begin{array}{c c c c}0&0&0&1\end{array}}$ & any & $W(2) + W(1) + V(2)$ \\

\def\arraystretch{0.5}\arraycolsep=0pt$e_{\tiny\begin{array}{c c c c}0&1&0&0\end{array}} + e_{\tiny\begin{array}{c c c c}0&0&0&1\end{array}} + e_{\tiny\begin{array}{c c c c}2&2&2&1\end{array}}$ & $p=2$ & $W(2) + V(2)^2$ \\
\endfirsthead
\def\arraystretch{0.5}\arraycolsep=0pt$e_{\tiny\begin{array}{c c c c}1&0&0&0\end{array}} + e_{\tiny\begin{array}{c c c c}0&0&1&0\end{array}} + e_{\tiny\begin{array}{c c c c}0&0&0&1\end{array}}$ & any & $W(2) + V(4)$ \\

\def\arraystretch{0.5}\arraycolsep=0pt$e_{\tiny\begin{array}{c c c c}0&0&0&1\end{array}}$ & any &
$W(1)^3 + V(2)$ \\

\def\arraystretch{0.5}\arraycolsep=0pt$e_{\tiny\begin{array}{c c c c}0&0&0&1\end{array}} + e_{\tiny\begin{array}{c c c c}2&2&2&1\end{array}}$ & $p=2$ & $W(1)^2 + V(2)^2$ \\

\def\arraystretch{0.5}\arraycolsep=0pt$e_{\tiny\begin{array}{c c c c}0&0&1&0\end{array}} + e_{\tiny\begin{array}{c c c c}0&0&0&1\end{array}}$ & any & $W(1)^2 + V(4)$ \\

\def\arraystretch{0.5}\arraycolsep=0pt$e_{\tiny\begin{array}{c c c c}0&0&1&0\end{array}} + e_{\tiny\begin{array}{c c c c}0&0&0&1\end{array}} + e_{\tiny\begin{array}{c c c c}2&2&2&1\end{array}}$ & any & $W(1) + V(4) + V(2)$ \\

\def\arraystretch{0.5}\arraycolsep=0pt$e_{\tiny\begin{array}{c c c c}0&1&0&0\end{array}} + e_{\tiny\begin{array}{c c c c}0&0&1&0\end{array}} + e_{\tiny\begin{array}{c c c c}0&0&0&1\end{array}}$ & any & $W(1) + V(6)$ \\

\def\arraystretch{0.5}\arraycolsep=0pt$e_{\tiny\begin{array}{c c c c}0&0&1&0\end{array}} + e_{\tiny\begin{array}{c c c c}0&0&0&1\end{array}} + e_{\tiny\begin{array}{c c c c}0&2&2&1\end{array}} + e_{\tiny\begin{array}{c c c c}2&2&2&1\end{array}}$ & $p=2$ & $V(4) + V(2)^2$ \\ 

\def\arraystretch{0.5}\arraycolsep=0pt$e_{\tiny\begin{array}{c c c c}1&0&0&0\end{array}} + e_{\tiny\begin{array}{c c c c}0&0&1&0\end{array}} + e_{\tiny\begin{array}{c c c c}0&0&0&1\end{array}} + e_{\tiny\begin{array}{c c c c}0&2&2&1\end{array}}$ & $p=2$ & $V(4)^2$ \\

\def\arraystretch{0.5}\arraycolsep=0pt$e_{\tiny\begin{array}{c c c c}0&1&0&0\end{array}} + e_{\tiny\begin{array}{c c c c}0&0&1&0\end{array}} + e_{\tiny\begin{array}{c c c c}0&0&0&1\end{array}} + e_{\tiny\begin{array}{c c c c}2&2&2&1\end{array}}$ & any & $V(6) + V(2)$ \\

\def\arraystretch{0.5}\arraycolsep=0pt$e_{\tiny\begin{array}{c c c c}1&0&0&0\end{array}} + e_{\tiny\begin{array}{c c c c}0&1&0&0\end{array}} + e_{\tiny\begin{array}{c c c c}0&0&1&0\end{array}} + e_{\tiny\begin{array}{c c c c}0&0&0&1\end{array}}$ & any & $V(8)$ \\
\hline
\def\arraystretch{0.5}\arraycolsep=0pt$e_{\tiny\begin{array}{c c c c}1&0&0&0\end{array}} + e_{\tiny\begin{array}{c c c c}0&1&0&0\end{array}} + e_{\tiny\begin{array}{c c c c}0&0&1&0\end{array}} + e_{\tiny\begin{array}{c c c c}2&2&2&1\end{array}}$ (nilp. only) & $p=2$ & $W_1(4)$ \\

\def\arraystretch{0.5}\arraycolsep=0pt$e_{\tiny\begin{array}{c c c c}0&1&0&0\end{array}} + e_{\tiny\begin{array}{c c c c}0&0&1&0\end{array}} + e_{\tiny\begin{array}{c c c c}0&2&2&1\end{array}}$ (nilp. only)  & $p=2$ & $W_1(3) + W(1)$ \\
 \hline
 \caption{Representatives for nilpotent orbits and unipotent classes of $\text{Sp}_8(k)$. \label{t:sp8reps}}
\end{longtable}\end{center}

\begin{center}\begin{longtable}{|l|l|l|}
\hline
 Representative $x$ & Char. & $V \downarrow k[u]$ \\
\hline 
$1$ & any & $W(1)^3 + V(1)$ \\&& $W(1)^3 + R$  \\
\def\arraystretch{0.5}\arraycolsep=0pt$x_{\tiny\begin{array}{c c c}1&0&0\end{array}}(1) x_{\tiny\begin{array}{c c c}0&1&0\end{array}}(1)$ & any & $W(3)+V(1)$ \\&& $W(3) + R$  \\
\def\arraystretch{0.5}\arraycolsep=0pt$x_{\tiny\begin{array}{c c c}1&0&0\end{array}}(1)$ & any & $W(2) + W(1)+V(1)$  \\&& $W(2) + W(1) + R$  \\
\def\arraystretch{0.5}\arraycolsep=0pt$x_{\tiny\begin{array}{c c c}1&0&0\end{array}}(1) x_{\tiny\begin{array}{c c c}0&0&1\end{array}}(1)$ & any & $W(2) + V(3)$ \\&& $W(2) + V(2) + R$ \\
\def\arraystretch{0.5}\arraycolsep=0pt$x_{\tiny\begin{array}{c c c}0&0&1\end{array}}(1)$ & any & $W(1)^2 + V(3)$  \\&& $W(1)^2 + V(2) + R$  \\
\def\arraystretch{0.5}\arraycolsep=0pt$x_{\tiny\begin{array}{c c c}0&1&0\end{array}}(1)  x_{\tiny\begin{array}{c c c}0&0&1\end{array}}(1)$ & any & $W(1) + V(5)$  \\&& $W(1) + V(4) + R$  \\
\def\arraystretch{0.5}\arraycolsep=0pt$x_{\tiny\begin{array}{c c c}1&0&0\end{array}}(1)  x_{\tiny\begin{array}{c c c}1&2&2\end{array}}(1)$ & $p=2$ & $W(1) + V(2)^2 + R$ \\
\def\arraystretch{0.5}\arraycolsep=0pt$x_{\tiny\begin{array}{c c c}1&0&0\end{array}}(1)  x_{\tiny\begin{array}{c c c}0&1&0\end{array}}(1)  x_{\tiny\begin{array}{c c c}0&1&2\end{array}}(1)$ & $p=2$ & $V(4) + V(2) + R$ \\
\def\arraystretch{0.5}\arraycolsep=0pt$x_{\tiny\begin{array}{c c c}1&0&0\end{array}}(1)  x_{\tiny\begin{array}{c c c}0&1&0\end{array}}(1)  x_{\tiny\begin{array}{c c c}0&0&1\end{array}}(1)$ & any & $V(7)$ \\&& $V(6) + R$  \\
\hline
\caption{Representatives for unipotent classes of $\text{SO}_7(k)$. \label{t:so7repsunip}}
\end{longtable}\end{center}

\begin{center}\begin{longtable}{|l|l|l|}
\hline
 Representative $x$ & Char. & $V \downarrow k[e]$ \\
\hline 
$0$ & any & $W(1)^3 + V(1)$  \\&& $W(1)^3 + D(1)$  \\
\def\arraystretch{0.5}\arraycolsep=0pt$e_{\tiny\begin{array}{c c c}1&0&0\end{array}} + e_{\tiny\begin{array}{c c c}0&1&0\end{array}}$ & any & $W(3)+V(1)$  \\&& $W(3) + D(1)$   \\
\def\arraystretch{0.5}\arraycolsep=0pt$e_{\tiny\begin{array}{c c c}1&0&0\end{array}}$ & any & $W(2) + W(1)+V(1)$  \\&& $W(2) + W(1) + D(1)$  \\
\def\arraystretch{0.5}\arraycolsep=0pt$e_{\tiny\begin{array}{c c c}1&0&0\end{array}} + e_{\tiny\begin{array}{c c c}0&0&1\end{array}}$ & any & $W(2) + V(3)$  \\&& $W(2) + D(2)$  \\
\def\arraystretch{0.5}\arraycolsep=0pt$e_{\tiny\begin{array}{c c c}0&0&1\end{array}}$ & any & $W(1)^2 + V(3)$ \\&& $W(1)^2 + D(2)$  \\
\def\arraystretch{0.5}\arraycolsep=0pt$e_{\tiny\begin{array}{c c c}0&1&0\end{array}} + e_{\tiny\begin{array}{c c c}0&0&1\end{array}}$ & any & $W(1) + V(5)$  \\&& $W(1) + D(3)$ \\
\def\arraystretch{0.5}\arraycolsep=0pt$e_{\tiny\begin{array}{c c c}1&0&0\end{array}} + e_{\tiny\begin{array}{c c c}1&2&2\end{array}}$ & $p=2$ & $W_2(2) + W(1) + D(1)$ \\

\def\arraystretch{0.5}\arraycolsep=0pt$e_{\tiny\begin{array}{c c c}1&0&0\end{array}} + e_{\tiny\begin{array}{c c c}0&1&0\end{array}} + e_{\tiny\begin{array}{c c c}0&1&2\end{array}}$ & $p=2$ & $W_3(3) + D(1)$ \\
\def\arraystretch{0.5}\arraycolsep=0pt$e_{\tiny\begin{array}{c c c}1&0&0\end{array}} + e_{\tiny\begin{array}{c c c}0&1&0\end{array}} + e_{\tiny\begin{array}{c c c}0&0&1\end{array}}$ & any & $V(7)$  \\&& $D(4)$   \\
\hline
\caption{Representatives for nilpotent orbits of $\text{SO}_7(k)$. \label{t:so7repsnil}}
\endfirsthead
\end{longtable}\end{center}

\begin{center}\begin{longtable}{|l|l|l|l|}
\hline
Representative $x$ & $V \downarrow k[e]$ & $V \downarrow k[u]$ \\
\hline 
 $0$ & $W(1)^4$ & $W(1)^4$ \\
\def\arraystretch{0.5}\arraycolsep=0pt$e_{\tiny\begin{array}{c c c }\multirow{2}{*}{1} & \multirow{2}{*}{0} & 0 \\ & & 0\end{array}} + e_{\tiny\begin{array}{c c c }\multirow{2}{*}{0} & \multirow{2}{*}{1} & 0 \\ & & 0\end{array}} + e_{\tiny\begin{array}{c c c }\multirow{2}{*}{0} & \multirow{2}{*}{0} & 1 \\ & & 0\end{array}}$
 &  $W(4)$ & $W(4)$ \\

\def\arraystretch{0.5}\arraycolsep=0pt$e_{\tiny\begin{array}{c c c }\multirow{2}{*}{1} & \multirow{2}{*}{0} & 0 \\ & & 0\end{array}} + e_{\tiny\begin{array}{c c c }\multirow{2}{*}{0} & \multirow{2}{*}{1} & 0 \\ & & 0\end{array}} + e_{\tiny\begin{array}{c c c }\multirow{2}{*}{0} & \multirow{2}{*}{0} & 0 \\ & & 1\end{array}}$
 &  $W(4)$ & $W(4)$ \\

 \def\arraystretch{0.5}\arraycolsep=0pt$e_{\tiny\begin{array}{c c c }\multirow{2}{*}{1} & \multirow{2}{*}{0} & 0 \\ & & 0\end{array}} + e_{\tiny\begin{array}{c c c }\multirow{2}{*}{0} & \multirow{2}{*}{1} & 0 \\ & & 0\end{array}}$ &  $W(3) + W(1)$ & $W(3) + W(1)$ \\ 

\def\arraystretch{0.5}\arraycolsep=0pt$e_{\tiny\begin{array}{c c c }\multirow{2}{*}{1} & \multirow{2}{*}{0} & 0 \\ & & 0\end{array}} + e_{\tiny\begin{array}{c c c }\multirow{2}{*}{0} & \multirow{2}{*}{0} & 1 \\ & & 0\end{array}}$ &  $W(2)^2$ & $W(2)^2$ \\

\def\arraystretch{0.5}\arraycolsep=0pt$e_{\tiny\begin{array}{c c c }\multirow{2}{*}{1} & \multirow{2}{*}{0} & 0 \\ & & 0\end{array}} + e_{\tiny\begin{array}{c c c }\multirow{2}{*}{0} & \multirow{2}{*}{0} & 0 \\ & & 1\end{array}}$ &  $W(2)^2$ & $W(2)^2$ \\

\def\arraystretch{0.5}\arraycolsep=0pt$e_{\tiny\begin{array}{c c c }\multirow{2}{*}{1} & \multirow{2}{*}{0} & 0 \\ & & 0\end{array}}$ & $W(2) + W(1)^2$ & $W(2) + W(1)^2$ \\

\def\arraystretch{0.5}\arraycolsep=0pt$e_{\tiny\begin{array}{c c c }\multirow{2}{*}{1} & \multirow{2}{*}{0} & 0 \\ & & 0\end{array}} + e_{\tiny\begin{array}{c c c }\multirow{2}{*}{0} & \multirow{2}{*}{0} & 1 \\ & & 0\end{array}} + e_{\tiny\begin{array}{c c c }\multirow{2}{*}{0} & \multirow{2}{*}{0} & 0 \\ & & 1\end{array}}$ & $W(2) + V(3) + V(1)$  & $W(2) + V(3) + V(1)$  \\&$W(2) + W_2(2)$ & $W(2) + V(2)^2$   \\

\def\arraystretch{0.5}\arraycolsep=0pt$e_{\tiny\begin{array}{c c c }\multirow{2}{*}{0} & \multirow{2}{*}{0} & 1 \\ & & 0\end{array}} + e_{\tiny\begin{array}{c c c }\multirow{2}{*}{0} & \multirow{2}{*}{0} & 0 \\ & & 1\end{array}}$ & 
$W(1)^2 + V(3) + V(1)$  & $W(1)^2 + V(3) + V(1)$  \\&$W(1)^2 + W_2(2)$ & $W(1)^2 + V(2)^2$  \\

\def\arraystretch{0.5}\arraycolsep=0pt$e_{\tiny\begin{array}{c c c }\multirow{2}{*}{0} & \multirow{2}{*}{1} & 0 \\ & & 0\end{array}} + e_{\tiny\begin{array}{c c c }\multirow{2}{*}{0} & \multirow{2}{*}{0} & 1 \\ & & 0\end{array}} + e_{\tiny\begin{array}{c c c }\multirow{2}{*}{0} & \multirow{2}{*}{0} & 0 \\ & & 1\end{array}}$ & $W(1) + V(5) + V(1)$  & $W(1) + V(5) + V(1)$  \\&$W(1) + W_3(3)$ & $W(1) + V(4) + V(2)$  \\

\def\arraystretch{0.5}\arraycolsep=0pt$e_{\tiny\begin{array}{c c c }\multirow{2}{*}{1} & \multirow{2}{*}{0} & 0 \\ & & 0\end{array}} + e_{\tiny\begin{array}{c c c }\multirow{2}{*}{0} & \multirow{2}{*}{1} & 0 \\ & & 0\end{array}} + e_{\tiny\begin{array}{c c c }\multirow{2}{*}{0} & \multirow{2}{*}{0} & 1 \\ & & 0\end{array}} + e_{\tiny\begin{array}{c c c }\multirow{2}{*}{1} & \multirow{2}{*}{1} & 0 \\ & & 1\end{array}}$ & $V(5)+V(3)$  & $V(5)+V(3)$  \\& $W_3(4)$ & $V(4)^2$  \\

\def\arraystretch{0.5}\arraycolsep=0pt$e_{\tiny\begin{array}{c c c }\multirow{2}{*}{1} & \multirow{2}{*}{0} & 0 \\ & & 0\end{array}} + e_{\tiny\begin{array}{c c c }\multirow{2}{*}{0} & \multirow{2}{*}{1} & 0 \\ & & 0\end{array}} + e_{\tiny\begin{array}{c c c }\multirow{2}{*}{0} & \multirow{2}{*}{0} & 1 \\ & & 0\end{array}} + e_{\tiny\begin{array}{c c c }\multirow{2}{*}{0} & \multirow{2}{*}{0} & 0 \\ & & 1\end{array}}$ &  $V(7)+V(1)$  & $V(7)+V(1)$  \\&$W_4(4)$ & $V(6) + V(2)$  \\

\hline
\caption{Representatives for nilpotent orbits and unipotent classes of $\text{SO}_{8}(k)$. \label{t:so8reps}}
\end{longtable}\end{center}

\begin{center}\begin{longtable}{|l|l|l|}
\hline
Representative $x$ & $V \downarrow k[e]$ & $V \downarrow k[u]$ \\
\hline 
 $0$ & $W(1)^5$ & $W(1)^5$ \\

\def\arraystretch{0.5}\arraycolsep=0pt$e_{\tiny\begin{array}{c c c c }\multirow{2}{*}{1} & \multirow{2}{*}{0} & \multirow{2}{*}{0} & 0 \\ & & & 0\end{array}} + e_{\tiny\begin{array}{c c c c }\multirow{2}{*}{0} & \multirow{2}{*}{1} & \multirow{2}{*}{0} & 0 \\ & & & 0\end{array}} + e_{\tiny\begin{array}{c c c c }\multirow{2}{*}{0} & \multirow{2}{*}{0} & \multirow{2}{*}{1} & 0 \\ & & & 0\end{array}} + e_{\tiny\begin{array}{c c c c }\multirow{2}{*}{0} & \multirow{2}{*}{0} & \multirow{2}{*}{0} & 1 \\ & & & 0\end{array}}$
 & $W(5)$ & $W(5)$ \\

\def\arraystretch{0.5}\arraycolsep=0pt$e_{\tiny\begin{array}{c c c c }\multirow{2}{*}{1} & \multirow{2}{*}{0} & \multirow{2}{*}{0} & 0 \\ & & & 0\end{array}} + e_{\tiny\begin{array}{c c c c }\multirow{2}{*}{0} & \multirow{2}{*}{1} & \multirow{2}{*}{0} & 0 \\ & & & 0\end{array}} + e_{\tiny\begin{array}{c c c c }\multirow{2}{*}{0} & \multirow{2}{*}{0} & \multirow{2}{*}{1} & 0 \\ & & & 0\end{array}}$
 & $W(4) + W(1)$ & $W(4) + W(1)$ \\

\def\arraystretch{0.5}\arraycolsep=0pt$e_{\tiny\begin{array}{c c c c }\multirow{2}{*}{1} & \multirow{2}{*}{0} & \multirow{2}{*}{0} & 0 \\ & & & 0\end{array}} + e_{\tiny\begin{array}{c c c c }\multirow{2}{*}{0} & \multirow{2}{*}{1} & \multirow{2}{*}{0} & 0 \\ & & & 0\end{array}} + e_{\tiny\begin{array}{c c c c }\multirow{2}{*}{0} & \multirow{2}{*}{0} & \multirow{2}{*}{0} & 1 \\ & & & 0\end{array}}$ & $W(3) + W(2)$ & $W(3) + W(2)$ \\ 

\def\arraystretch{0.5}\arraycolsep=0pt$e_{\tiny\begin{array}{c c c c }\multirow{2}{*}{1} & \multirow{2}{*}{0} & \multirow{2}{*}{0} & 0 \\ & & & 0\end{array}} + e_{\tiny\begin{array}{c c c c }\multirow{2}{*}{0} & \multirow{2}{*}{1} & \multirow{2}{*}{0} & 0 \\ & & & 0\end{array}}$ & $W(3) + W(1)^2$ & $W(3) + W(1)^2$ \\ 

\def\arraystretch{0.5}\arraycolsep=0pt$e_{\tiny\begin{array}{c c c c }\multirow{2}{*}{1} & \multirow{2}{*}{0} & \multirow{2}{*}{0} & 0 \\ & & & 0\end{array}} + e_{\tiny\begin{array}{c c c c }\multirow{2}{*}{0} & \multirow{2}{*}{0} & \multirow{2}{*}{1} & 0 \\ & & & 0\end{array}}$  & $W(2)^2 + W(1)$ & $W(2)^2 + W(1)$ \\

\def\arraystretch{0.5}\arraycolsep=0pt$e_{\tiny\begin{array}{c c c c }\multirow{2}{*}{1} & \multirow{2}{*}{0} & \multirow{2}{*}{0} & 0 \\ & & & 0\end{array}}$  & $W(2) + W(1)^3$ & $W(2) + W(1)^3$ \\

\def\arraystretch{0.5}\arraycolsep=0pt$e_{\tiny\begin{array}{c c c c }\multirow{2}{*}{1} & \multirow{2}{*}{0} & \multirow{2}{*}{0} & 0 \\ & & & 0\end{array}} + e_{\tiny\begin{array}{c c c c }\multirow{2}{*}{0} & \multirow{2}{*}{1} & \multirow{2}{*}{0} & 0 \\ & & & 0\end{array}} + e_{\tiny\begin{array}{c c c c }\multirow{2}{*}{0} & \multirow{2}{*}{0} & \multirow{2}{*}{0} & 1 \\ & & & 0\end{array}} + e_{\tiny\begin{array}{c c c c }\multirow{2}{*}{0} & \multirow{2}{*}{0} & \multirow{2}{*}{0} & 0 \\ & & & 1\end{array}}$ & $W(3) + V(3) + V(1)$  & $W(3) + V(3) + V(1)$  \\&$W(3) + W_2(2)$ & $W(3) + V(2)^2$   \\

\def\arraystretch{0.5}\arraycolsep=0pt$e_{\tiny\begin{array}{c c c c }\multirow{2}{*}{1} & \multirow{2}{*}{0} & \multirow{2}{*}{0} & 0 \\ & & & 0\end{array}} + e_{\tiny\begin{array}{c c c c }\multirow{2}{*}{0} & \multirow{2}{*}{0} & \multirow{2}{*}{0} & 1 \\ & & & 0\end{array}} + e_{\tiny\begin{array}{c c c c }\multirow{2}{*}{0} & \multirow{2}{*}{0} & \multirow{2}{*}{0} & 0 \\ & & & 1\end{array}}$  & $W(2) + W(1) + V(3) + V(1)$  & $W(2) + W(1) + V(3) + V(1)$  \\&$W(2) + W(1) + W_2(2)$ & $W(2) + W(1) + V(2)^2$  \\

\def\arraystretch{0.5}\arraycolsep=0pt$e_{\tiny\begin{array}{c c c c }\multirow{2}{*}{0} & \multirow{2}{*}{0} & \multirow{2}{*}{0} & 1 \\ & & & 0\end{array}} + e_{\tiny\begin{array}{c c c c }\multirow{2}{*}{0} & \multirow{2}{*}{0} & \multirow{2}{*}{0} & 0 \\ & & & 1\end{array}}$ & $W(1)^3 + V(3) + V(1)$  & $W(1)^3 + V(3) + V(1)$  \\&$W(1)^3 + W_2(2)$ & $W(1)^3 + V(2)^2$  \\

\def\arraystretch{0.5}\arraycolsep=0pt$e_{\tiny\begin{array}{c c c c }\multirow{2}{*}{1} & \multirow{2}{*}{0} & \multirow{2}{*}{0} & 0 \\ & & & 0\end{array}} + e_{\tiny\begin{array}{c c c c }\multirow{2}{*}{0} & \multirow{2}{*}{0} & \multirow{2}{*}{1} & 0 \\ & & & 0\end{array}} + e_{\tiny\begin{array}{c c c c }\multirow{2}{*}{0} & \multirow{2}{*}{0} & \multirow{2}{*}{0} & 1 \\ & & & 0\end{array}} + e_{\tiny\begin{array}{c c c c }\multirow{2}{*}{0} & \multirow{2}{*}{0} & \multirow{2}{*}{0} & 0 \\ & & & 1\end{array}}$ & $W(2) + V(5) + V(1)$  & $W(2) + V(5) + V(1)$  \\&$W(2) + W_3(3)$ & $W(2) + V(4) + V(2)$  \\

\def\arraystretch{0.5}\arraycolsep=0pt$e_{\tiny\begin{array}{c c c c }\multirow{2}{*}{0} & \multirow{2}{*}{0} & \multirow{2}{*}{1} & 0 \\ & & & 0\end{array}} + e_{\tiny\begin{array}{c c c c }\multirow{2}{*}{0} & \multirow{2}{*}{0} & \multirow{2}{*}{0} & 1 \\ & & & 0\end{array}} + e_{\tiny\begin{array}{c c c c }\multirow{2}{*}{0} & \multirow{2}{*}{0} & \multirow{2}{*}{0} & 0 \\ & & & 1\end{array}}$ & $W(1)^2 + V(5) + V(1)$  & $W(1)^2 + V(5) + V(1)$  \\&$W(1)^2 + W_3(3)$ & $W(1)^2 + V(4) + V(2)$  \\

\def\arraystretch{0.5}\arraycolsep=0pt$e_{\tiny\begin{array}{c c c c }\multirow{2}{*}{0} & \multirow{2}{*}{1} & \multirow{2}{*}{0} & 0 \\ & & & 0\end{array}} + e_{\tiny\begin{array}{c c c c }\multirow{2}{*}{0} & \multirow{2}{*}{0} & \multirow{2}{*}{1} & 0 \\ & & & 0\end{array}} + e_{\tiny\begin{array}{c c c c }\multirow{2}{*}{0} & \multirow{2}{*}{0} & \multirow{2}{*}{0} & 1 \\ & & & 0\end{array}} + e_{\tiny\begin{array}{c c c c }\multirow{2}{*}{0} & \multirow{2}{*}{1} & \multirow{2}{*}{1} & 0 \\ & & & 1\end{array}}$ & $W(1) + V(5)+V(3)$  & $W(1) + V(5)+V(3)$  \\&$W(1) + W_3(4)$ & $W(1) + V(4)^2$  \\

\def\arraystretch{0.5}\arraycolsep=0pt$e_{\tiny\begin{array}{c c c c }\multirow{2}{*}{0} & \multirow{2}{*}{1} & \multirow{2}{*}{0} & 0 \\ & & & 0\end{array}} + e_{\tiny\begin{array}{c c c c }\multirow{2}{*}{0} & \multirow{2}{*}{0} & \multirow{2}{*}{1} & 0 \\ & & & 0\end{array}} + e_{\tiny\begin{array}{c c c c }\multirow{2}{*}{0} & \multirow{2}{*}{0} & \multirow{2}{*}{0} & 1 \\ & & & 0\end{array}} + e_{\tiny\begin{array}{c c c c }\multirow{2}{*}{0} & \multirow{2}{*}{0} & \multirow{2}{*}{0} & 0 \\ & & & 1\end{array}}$ & $W(1) + V(7)+V(1)$  & $W(1) + V(7)+V(1)$  \\&$W(1) + W_4(4)$ & $W(1) + V(6) + V(2)$  \\

\def\arraystretch{0.5}\arraycolsep=0pt$e_{\tiny\begin{array}{c c c c }\multirow{2}{*}{1} & \multirow{2}{*}{0} & \multirow{2}{*}{0} & 0 \\ & & & 0\end{array}} + e_{\tiny\begin{array}{c c c c }\multirow{2}{*}{0} & \multirow{2}{*}{1} & \multirow{2}{*}{0} & 0 \\ & & & 0\end{array}} + e_{\tiny\begin{array}{c c c c }\multirow{2}{*}{0} & \multirow{2}{*}{0} & \multirow{2}{*}{1} & 0 \\ & & & 0\end{array}} + e_{\tiny\begin{array}{c c c c }\multirow{2}{*}{0} & \multirow{2}{*}{0} & \multirow{2}{*}{0} & 1 \\ & & & 0\end{array}} + e_{\tiny\begin{array}{c c c c }\multirow{2}{*}{0} & \multirow{2}{*}{1} & \multirow{2}{*}{1} & 0 \\ & & & 1\end{array}}$ & $V(7)+V(3)$  & $V(7)+V(3)$  \\&$W_4(5)$ & $W(1) + V(6) + V(4)$  \\

\def\arraystretch{0.5}\arraycolsep=0pt$e_{\tiny\begin{array}{c c c c }\multirow{2}{*}{1} & \multirow{2}{*}{0} & \multirow{2}{*}{0} & 0 \\ & & & 0\end{array}} + e_{\tiny\begin{array}{c c c c }\multirow{2}{*}{0} & \multirow{2}{*}{1} & \multirow{2}{*}{0} & 0 \\ & & & 0\end{array}} + e_{\tiny\begin{array}{c c c c }\multirow{2}{*}{0} & \multirow{2}{*}{0} & \multirow{2}{*}{1} & 0 \\ & & & 0\end{array}} + e_{\tiny\begin{array}{c c c c }\multirow{2}{*}{0} & \multirow{2}{*}{0} & \multirow{2}{*}{0} & 1 \\ & & & 0\end{array}} + e_{\tiny\begin{array}{c c c c }\multirow{2}{*}{0} & \multirow{2}{*}{0} & \multirow{2}{*}{0} & 0 \\ & & & 1\end{array}}$ & $W(1) + V(9)+V(1)$  & $W(1) + V(9)+V(1)$  \\&$W_5(5)$ & $V(8) + V(2)$  \\

 \hline
 \caption{Representatives for nilpotent orbits and unipotent classes of $\text{SO}_{10}(k)$. \label{t:so10reps}}
\end{longtable}\end{center}

\section{Proof of Theorem \ref{p:nilprec}} \label{s:proofrecs}

For unipotent elements, the required result is given in \cite[Section 3]{L09}. We describe the calculations in the nilpotent case. First we need the following lemma, which shows that for the action of nilpotent elements on $\Lie(G)$, the coincidences of Jordan block sizes for nilpotent elements on different orbits are the same in the adjoint case and the simply connected case.

\begin{lemma}\label{lemma:exadsc}
Let $G_{sc}$ be a simply connected simple algebraic group over $k$ of exceptional type, and let $G_{ad}$ be the corresponding simple algebraic group of adjoint type. Furthermore, let $x$ be either a unipotent element of $G_{sc}$ or a nilpotent element of $\Lie(G_{sc})$. Then the Jordan block sizes of $x$ on $\Lie(G_{sc})$ and $\Lie(G_{ad})$ are the same.
\end{lemma}

\begin{proof}For types $G_2$, $F_4$ and $E_8$ we have $G_{sc} = G_{ad}$, so it suffices to consider the case where $G_{sc}$ is of type $E_6$ or $E_7$. In this case, we show that $\Lie(G_{ad}) \cong \Lie(G_{sc})^*$ as $G_{sc}$-modules, from which the lemma follows. If $\Lie(G_{sc})$ is simple, then $\Lie(G_{sc})$ is an irreducible Weyl module and thus $\Lie(G_{sc})^* \cong \Lie(G_{sc})$ and $\Lie(G_{ad}) \cong \Lie(G_{sc})$ by \cite[Lemma II.2.13 b)]{Jan03}. 

If $\Lie(G_{sc})$ is not simple ($p = 3$ for $E_6$ and $p = 2$ for $E_7$), then by \cite{Hog82} there is a nonsplit short exact sequence $$0 \rightarrow k \rightarrow \Lie(G_{sc}) \rightarrow W \rightarrow 0$$ of $G_{sc}$-modules, where $W$ is irreducible. Since $W$ is self-dual, it follows that we have a nonsplit short exact sequence $$0 \rightarrow W \rightarrow \Lie(G_{sc})^* \rightarrow k \rightarrow 0.$$ Because $\Lie(G_{sc})$ is a Weyl module with highest weight the highest root, by \cite[II.2.12 (4), Proposition II.2.14]{Jan03} we have $\Ext_{G_{sc}}^1(k, W) \cong \Ext_{G_{sc}}^1(W, k) \cong k$. Furthermore, the group $\Ext_{G_{sc}}^1(k, W)$ is a $k$-vector space, and equivalence classes of extensions which are scalar multiples of each other correspond to isomorphic $G_{sc}$-modules. It follows then from $\Ext_{G_{sc}}^1(k, W) \cong k$ that a nonsplit extension of $k$ by $W$ is unique up to isomorphism of $G_{sc}$-modules, and so must be isomorphic to $\Lie(G_{sc})^*$. 

By \cite{Hog82} there is a nonsplit short exact sequence $$0 \rightarrow W \rightarrow \Lie(G_{ad}) \rightarrow k \rightarrow 0$$ of $G_{sc}$-modules, so we conclude that $\Lie(G_{ad}) \cong \Lie(G_{sc})^*$ as $G_{sc}$-modules.\end{proof}

The representatives and tables of Jordan blocks on the minimal and adjoint modules is given in \cite{Ste16}. The remaining data is supplied via in-built functions in  GAP4 \cite{GAP} and \textsc{Magma} such as \texttt{LieDerivedSeries}. It is then straightforward to confirm there are no coincidences.

\section{Proof of Theorem \ref{nilreps}}

We first treat the case where $G$ is classical, making heavy use of Lemma \ref{l:decomps}. Throughout the section $x$ and $x'$ are unipotent elements of $G$ or nilpotent elements of $\text{Lie}(G)$. In light of Lemma \ref{l:isogeny}, we take $G$ to be $\text{SL}(V)$, $\text{Sp}(V)$, or $\text{SO}(V)$ accordingly with natural module $V$. The proof is completed in the following two lemmas. 

\begin{lemma}
Suppose that $x$ is not conjugate to an element in Table \ref{t:classnil} or Table \ref{t:classunip}. Then $x$ is not eminent. 
\end{lemma}

\begin{proof}
If $x$ is not distinguished, then $x$ is contained in a Levi subgroup or a Levi subalgebra and thus is not eminent. Therefore in what follows, we will assume that $x$ is distinguished. This already completes the proof for $G = \SL(V)$, since the only distinguished elements are the regular ones. Suppose then that $G = \text{Sp}(V)$ or $G = \text{SO}(V)$. Comparing Tables \ref{t:classnil} and \ref{t:classunip} with Lemma \ref{l:decomps}, we can assume that $V \downarrow k[x]$ has at least two summands, and $r=0$ in the notation of Lemma \ref{l:decomps}. 

If $p \neq 2$, this means that $V \downarrow k[x] = \sum\limits_{j=1}^{s} V(n_j)$ with $s \geq 2$. It follows that $x$ is contained in the stabiliser of a non-trivial decomposition of $V$ into non-degenerate subspaces, and thus $x$ is contained in a proper generalised subsystem subgroup or subalgebra. 

For $p=2$, we must consider the unipotent and nilpotent cases separately. Suppose first that $x$ is unipotent. If $G$ is of type $C_n$, then the argument from the previous paragraph applies. 

If $G$ has type $D_n$, then by Lemma \ref{l:decomps}(ii) we have $V \downarrow k[x] = \sum\limits_{j=1}^{s} V(2n_j)$, where $s$ is even. When there are at least four summands, it follows that $x$ is contained in a maximal subsystem subgroup of type $D_m D_{n-m}$ (one may choose $m = n_1 + n_2$). For $s = 2$ there is one case remaining which is not in Table \ref{t:classunip}, namely $V \downarrow k[x] = V(2n-2) + V(2)$, corresponding to the regular unipotent class. It follows from the description of unipotent classes of $B_{n-1}$ in \cite[Section~6.8]{LS12} that $x$ is contained in the generalised subsystem subgroup of type $B_{n-1}$ (conjugate to the element $x'$ in $\SO(W)$ with $W \downarrow k[x'] = V(2n-2) \oplus R$). 

If $G$ has type $B_n$, then in the notation of Lemma \ref{l:decomps}(iii) we have $V \downarrow k[x] = V \downarrow k[x] = \sum\limits_{j=1}^{s} V(2n_j) + R$, where $s > 1$. Similarly to the previous cases, it follows that $x$ is contained in a subsystem subgroup of type $D_m B_{n-m-1}$ (again one may choose $m = n_1 + n_2$).

It remains to consider the case where $x$ is nilpotent. If $G$ has type $C_n$, then $V \downarrow k[x]$ has at least two summands, each of them isomorphic to $V(2q_k)$ or $W_{l_j}(n_j)$. Therefore $x$ is contained in the stabiliser of the non-trivial decomposition of $V$ into two non-degenerate subspaces, which is a subsystem subalgebra of type $C_m C_{n-m}$. If $G$ has type $D_n$, then $V \downarrow k[x]$ either has at least two summands isomorphic to $W_{l_j}(n_j)$, or is isomorphic to $W_n(n)$. In the first case, $x$ is contained in a subsystem subalgebra of type $D_{n_j} D_{n-n_j}$. In the second case, it follows from \cite[Section~5.6]{LS12} that $x$ is contained in the generalised subsystem subalgebra of type $B_{n-1}$ (conjugate to the element $x'$ in $\so(W)$ with $W \downarrow k[x'] = D(n)$). Finally if $G$ has type $B_n$, then $V \downarrow k[x]$ has at least one summand isomorphic to $W_{l_j}(n_j)$, and so $x$ is contained in a subsystem subalgebra of type $D_{n_j} B_{n-n_j}$.\end{proof}

\begin{lemma}
The elements in Tables \ref{t:classnil} and \ref{t:classunip} are eminent. 
\end{lemma}

\begin{proof}
Let $x$ be such an element. It follows from Lemma \ref{l:decomps} that $x$ is distinguished. Therefore if $x$ is not eminent, then it must be contained one of the maximal generalised subsystem subgroups or subalgebras described in Lemma \ref{l:maxgensubsys}. 

Suppose first that $x$ acts on $V$ with a single Jordan block. Elements in the subgroups and subalgebras of Lemma \ref{l:maxgensubsys} act on $V$ with at least two Jordan blocks, so it follows that $x$ is eminent. The remaining cases occur when $p=2$ and in all of them $x$ acts with two Jordan blocks on $V$. We treat them in turn. 

Consider first the case where $G = \SO(V)$ is of type $B_n$, and let $R$ be the $1$-dimensional radical of $V$. We will show that elements from the subgroups and subalgebras of Lemma \ref{l:maxgensubsys} act on $V$ with at least three Jordan blocks, which proves that $x$ must be eminent. For $M < G$ of type $D_n$, we have $V \downarrow M = W \oplus R$, where $W$ is the natural module for $M = \SO(W)$. Unipotent elements of $M$ and nilpotent elements of $\text{Lie}(M)$ act on $W$ with an even number of Jordan blocks by Lemma \ref{l:decomps} (ii) and (v), so in particular they have at least three Jordan blocks on $V$. The other possibility in Lemma \ref{l:maxgensubsys} is $M < G$ of type $B_mB_{n-m}$ for $1 \leq m \leq n/2$. In this case $(V/R) \downarrow M = V_1 \oplus V_2$, where $V_1$ is $2m$-dimensional and $V_2$ is $2(n-m)$-dimensional. Thus unipotent elements of $M$ and nilpotent elements of $\text{Lie}(M)$ act on $V/R$ with at least three Jordan blocks by Lemma \ref{l:decomps} (iii) and (vi), and therefore they also act on $V$ with at least three Jordan blocks.

If $G = \Sp(V)$ of type $C_n$, then $x$ is nilpotent and $V \downarrow k[x] = W_l(n)$ for some $0 < l <\frac{n}{2}$. Such an element $x$ acts with two Jordan blocks of size $n$ on $V$. By considering each of the maximal subsystem subalgebras, we see that only those of type $C_{\frac{n}{2}}C_{\frac{n}{2}}$ (if $n$ is even) and $D_n$ contain elements that act with two Jordan blocks of size $n$. In type $C_{\frac{n}{2}}C_{\frac{n}{2}}$ the only element which does that acts as $V(n) + V(n)$ on $V$, which is already in canonical form and thus not conjugate to $x$. In the second case, the elements $x'$ which act with two Jordan blocks in $\SO(V')$ have $V' \downarrow k[x'] = W_l(n)$ for some $ \frac{n+1}{2} < l \leq n$. The definition of $W_l(n)$ for orthogonal groups in even dimension, given in \cite[p. 66]{LS12}, shows that all of these elements $x'$ are conjugate to a nilpotent element acting via $W(n)$ in $\Lie(G)$, and thus not conjugate to $x$.      

If $G = \SO(V)$ of type $D_n$, then there are both unipotent classes and nilpotent orbits to consider. In all of the cases the elements $x$ being considered act with two Jordan blocks on $V$. Using Lemma \ref{l:decomps} we see that every element contained in subsystem subgroups and subalgebras of type $D_m D_{n-m}$ ($1 \leq m \leq \frac{m}{2}$) will act with at least four Jordan blocks on $V$. Thus they are not conjugate to $x$. By Lemma \ref{l:maxgensubsys}, what remains is to show that $x$ is not conjugate to an element of the maximal generalised subsystem subgroup $M$ (or subalgebra $\m$) of type $B_{n-1}$. When $x$ is unipotent, this follows from \cite[Lemma~3.8]{MK18}, which shows that the only unipotent elements $x' \in M$ that act with two Jordan blocks on $V$ are those with $V \downarrow k[x'] = V(2) + V(2n-2)$. Such elements $x'$ are regular and not conjugate to $x$. Now let $x$ be nilpotent. We consider the distinguished nilpotent elements $x'$ in $\m = \so(W)$. By Lemma \ref{l:decomps} (vi), we have $W \downarrow k[x'] = \sum_{j=1}^s W_{l_j}(n_j) \oplus D(m)$ for some integers $l_j$, $n_j$ and $m$. If $s=0$, then $W \downarrow k[x'] = D(n)$ and $x'$ is a regular nilpotent element of $\Lie(G)$, and therefore is not conjugate to $x$. When $s \geq 1$, the elements $x'$ are contained in generalised subsystem subalgebras of type $D_{n-m} B_{m-1}$, which are contained in maximal subsystem subalgebras of type $D_{n-m} D_{m}$. We have already seen that any nilpotent element contained in such a subalgebra will act with at least four Jordan blocks on $V$, and thus cannot be conjugate to $x$. This completes the proof of the lemma.   \end{proof}

The remaining task for $G$ of classical type is to prove that the representatives of the eminent elements do indeed act with the claimed decompositions on $V$. When $x$ is regular, a representative is provided by Remark \ref{regularelts}. The non-regular eminent classes occur in characteristic $p = 2$ for types $C_n$ and $D_n$, the representatives for these classes are constructed in the following sections. 

\subsection{Nilpotent representatives in type $C_n$} \label{s:typeC}

Let $G=\text{Sp}(V)$ with $V$ a $k$-vector space of dimension $2n$. We need to prove that the nilpotent elements $$e_l = \left(\sum_{i=1}^{n-1} e_{\alpha_i}\right)+  e_{2 \alpha_l + \cdots + 2 \alpha_{n-1} + \alpha_n}$$ in Table \ref{t:classnil} act as $W_l(n)$ on the natural module $V$. We do this in the next lemma, but first present our choice of Chevalley basis for $\sp(V)$. 

Recall the form described at the beginning of Section \ref{s:classical}. For this choice of form, one checks that there is a Cartan subalgebra $\mathfrak{h}$ of diagonal matrices in $\mathfrak{sp}(V)\subset \gl(V)$ of the form $\diag(h_1, \ldots, h_n, -h_n, \ldots, -h_1)$. For $1 \leq i \leq n$, define maps $\varepsilon_i : \mathfrak{h} \rightarrow k$ by $\varepsilon_i(h) = h_i$. 

For all $i,j$ let $E_{i,j}$ be the linear endomorphism on $V$ such that $E_{i,j}(e_j) = e_i$ and $E_{i,j}(e_k) = 0$ for $k \neq j$. Then one checks that the endomorphisms $e_\lambda$ in Table \ref{chevbassp4} are elements of $\mathfrak{sp}(V)$, and are simultaneous eigenvectors for $\ad \h$.

\begin{table}[!htbp] \begin{align*}e_{\varepsilon_i - \varepsilon_j} &= E_{i,j} - E_{n+j,n+i} & \text{ for all } i \neq j, \\
e_{\varepsilon_i + \varepsilon_j} &= E_{j, n+i} + E_{i, n+j} & \text{ for all } i \neq j, \\
e_{-(\varepsilon_i + \varepsilon_j)} &= E_{n+j, i} + E_{n+i, j} & \text{ for all } i \neq j, \\
e_{2 \varepsilon_i} &= E_{i, n+i} & \text{ for all } i, \\
e_{-2 \varepsilon_i} &= E_{n+i, i} & \text{ for all } i. 
\end{align*}\caption{Chevalley basis for $\mathfrak{sp}_{2n}$\label{chevbassp4}}
\end{table}

From a dimension count, we deduce that concatenating the elements in Table \ref{chevbassp4} together with a basis of $\h$ gives a basis of $\sp_{2n}$. Moreover, $\Phi = \{ \pm(\varepsilon_i \pm \varepsilon_j) : 1 \leq i < j \leq n \} \cup \{ \pm 2\varepsilon_i : 1 \leq i \leq n \}$, and $\Phi^+ = \{ \varepsilon_i \pm \varepsilon_j : 1 \leq i < j \leq n \} \cup \{ 2 \varepsilon_i : 1 \leq i \leq n \}$ is a system of positive roots. We let $\Delta = \{\alpha_1, \ldots, \alpha_n\}$ be our base of $\Phi$ corresponding to $\Phi^+$, where $\alpha_i = \varepsilon_i - \varepsilon_{i+1}$ for $1 \leq i < n$ and $\alpha_n = 2 \varepsilon_n$. We give the expressions of roots $\alpha \in \Phi^+$ in terms of $\Delta$ in Table \ref{table:rootexpressionsC}. Lastly, checking commutator relations reveals the basis in Table \ref{chevbassp4} is in fact a Chevalley basis, with positive structure constants for extraspecial pairs.

\begin{table}[!htbp]\begin{tabular}{| c l | l |}
\hline
Root & & $\alpha = \sum_{k = 1}^n c_k \alpha_i$ \\
\hline
$\varepsilon_i - \varepsilon_j$, & $1 \leq i < j \leq n$ & $\sum_{i \leq k \leq j-1} \alpha_k$ \\
$\varepsilon_i + \varepsilon_j$, & $1 \leq i < j \leq n$ & $\sum_{i \leq k \leq j-1} \alpha_k + \sum_{j \leq k \leq n-1} 2 \alpha_k + \alpha_n$ \\
$2 \varepsilon_i$, & $1 \leq i \leq n$ & $\sum_{i \leq k \leq n-1} 2 \alpha_k + \alpha_n$ \\
\hline 

\end{tabular}
\caption{Type $C_n$, expressions for positive roots in terms of base $\Delta$.}\label{table:rootexpressionsC}
\end{table}

In particular, $e_{\alpha_i} = E_{i,i+1} - E_{n+i+1,n+i}$ for $1 \leq i \leq n-1$, $e_{\alpha_n} = E_{n,2n}$ and $e_{2 \alpha_l + \cdots + 2 \alpha_{n-1} + \alpha_n} = E_{l,n+l}$ for $l < n$.  

\begin{lemma} \label{l:Cnilprep}
For $0<l<n/2$, the element $e_l = \sum_{i=1}^{n-1} e_{\alpha_i} + e_{2 \alpha_l + \cdots + 2 \alpha_{n-1} + \alpha_n} \in \sp_{2n}(k)$ acts on $V$ as $V \downarrow k[e_l] = W_l(n)$. 
\end{lemma}

\begin{proof}
For ease of notation we set $e = e_l$. From the description of the $k[e]$-module $W_l(n)$ in Section \ref{s:classical}, it suffices to show that $e$ acts on the natural $2n$-dimensional module $V$ with two Jordan blocks of size $n$ and that $l = \min\{ m : (e^{m+1}(v), {e}^{m}(v)) = 0 \textrm{ for all } v \in V\}$. 

Using Tables \ref{chevbassp4} and \ref{table:rootexpressionsC}, we write
$$e = E_{l,n+l} + \sum_{i=1}^{n-1} (E_{i,i+1} + E_{n+i+1,n+i}),$$
and therefore the action of $e$ on $V$ is described as follows. We have 
\begin{align*}
e (v_1) = & \ 0 \\
e (v_i) = & \ v_{i-1}\qquad (2 \leq i \leq n) \\
e (v_{n+l}) = & \ v_{l} + v_{n+l+1} \\
e (v_j) = & \ v_{j+1}\qquad (n+1 \leq j < 2n, j \neq n+l) \\
e (v_{2n}) = & \ 0. 
\end{align*} 
We immediately see that $e^n(v_i) = 0$ for all $1 \leq i \leq 2n$. Furthermore, the relations above readily imply that the kernel of $e$ is the $2$-dimensional subspace of $V$ generated by $v_1$ and $v_{2n}$. Thus the Jordan normal form of $e$ has two Jordan blocks of size $n$. 

To show that $l = \min\{ m : (e^{m+1}(v), {e}^{m}(v)) = 0 \textrm{ for all } v \in V\}$ we start by observing 
$$(e^{j+1} (v_{n+l-j}), e^{j}(v_{n+l-j})) = (v_l + v_{n+l+1}, v_{n+l}) =  1 \text{ for all } j < l.$$

It remains to show that $({e}^{l+1}(v), {e}^{l} (v)) = 0$ for all $v \in V$. Since $({e}^{l+1}(v_i), {e}^{l} (v_j)) + ({e}^{l+1}(v_j), {e}^{l} (v_i))=0$ for all $i,j$, by definition of $\sp(V)$ preserving the alternating form, it suffices to show that $({e}^{l+1}(v_i), {e}^{l} (v_i)) = 0$ for all basis vectors $v_i$. 

If $1 \leq i \leq n$ or $n+l < i \leq 2n$ we immediately have that $({e}^{l+1}(v_i), {e}^{l} (v_i)) = 0$ since $e$ stabilises the totally isotropic subspaces $\langle v_1, \ldots, v_n \rangle$ and $\langle v_{n+l+1}, \ldots, v_{2n} \rangle$. Now suppose that $1 \leq j \leq l$. Then $e^{l}(v_{n+j}) = v_{l-j+1} + v_{n+l+j}$ and so $e^{l+1}(v_{n+j}) = v_{l-j} + v_{n+l+j+1}$ and so $({e}^{l+1}(v_{n+j}), {e}^{l} (v_{n+j})) = 0$, as required. 
\end{proof}

\subsection{Nilpotent and unipotent representatives in type $D_n$} \label{s:typeD}
Let $G=\text{SO}(V)$ with $V$ a $k$-vector space of dimension $2n$. We need to establish the correctness of the representatives in Tables \ref{t:classunip} and \ref{t:nilprec}. Recall we define the root $\beta_l = \alpha_{n} + \sum_{i=2l-n-1}^{n-2} \alpha_{i}$ for $(n+1)/2 < l \leq n$.

We mimic the process from the last section. Using the form from the beginning of Section \ref{s:classical} we check we may choose a  Cartan subalgebra  $\mathfrak{h}$  of the form $\diag(h_1, \ldots, h_n, -h_n, \ldots, -h_1)$ with $h_i\in k$. For $1 \leq i \leq n$, define maps $\varepsilon_i : \mathfrak{h} \rightarrow k$ by $\varepsilon_i(h) = h_i$. A Chevalley basis for $\mathfrak{so}(V)$ with positive structure constants for extraspecial pairs is given in Table \ref{chevbasson}.

\begin{table}[!htbp]\begin{align*} 
e_{\varepsilon_i - \varepsilon_j} &= E_{i,j} - E_{n+j,n+i} & \text{ for all } i \neq j, \\
e_{\varepsilon_i + \varepsilon_j} &= E_{j,n+i} - E_{i,n+j} & \text{ for all } i < j, \\
e_{-(\varepsilon_i + \varepsilon_j)} &= E_{n+i,j} - E_{n+j,i} & \text{ for all } i < j.
\end{align*} \caption{Chevalley basis of $\so(2n)$.}\label{chevbasson}
\end{table}

Now $\Phi = \{ \pm (\varepsilon_i \pm \varepsilon_j) : 1 \leq i < j \leq n\}$ is the root system of $\mathfrak{so}(V)$, and $\Phi^+ = \{ \varepsilon_i \pm \varepsilon_j : 1 \leq i < j \leq n \}$ is a system of positive roots. Here the base $\Delta$ of $\Phi$ corresponding to $\Phi^+$ is $\Delta = \{ \alpha_1, \ldots, \alpha_n \}$, where $\alpha_i = \varepsilon_i - \varepsilon_{i+1}$ for $1 \leq i < n$ and $\alpha_n = \varepsilon_{n-1} + \varepsilon_n$. We give the expressions of roots $\alpha \in \Phi^+$ in terms of $\Delta$ in Table \ref{table:rootexpressionsD}.

\begin{table}[!htbp]\begin{tabular}{| cl | l |}
\hline
Root & & $\alpha = \sum_{k = 1}^n c_k \alpha_i$ \\
\hline
$\varepsilon_i - \varepsilon_j$, & $1 \leq i < j \leq n$ & $\sum_{i \leq k \leq j-1} \alpha_k$ \\
$\varepsilon_i + \varepsilon_j$, & $1 \leq i < j \leq n-1$ & $\sum_{i \leq k \leq j-1} \alpha_k + \sum_{j \leq k \leq n-2} 2 \alpha_k + \alpha_{n-1} + \alpha_n$ \\
$\varepsilon_i + \varepsilon_n$, & $1 \leq i \leq n-1$ & $\sum_{i \leq k \leq n-2} \alpha_k + \alpha_n$ \\
\hline
\end{tabular}
\caption{Type $D_n$, expressions for positive roots in terms of base $\Delta$.}\label{table:rootexpressionsD}
\end{table}

\begin{lemma} \label{l:Duniprep}
The element $u_l = \prod_{i=1}^{n-1} x_{\alpha_i}(1) \cdot x_{\beta_l}(1)$ acts on $V$ as $V\downarrow k[u_l]=V(2l-2) \oplus V(2n-2l+2)$.
\end{lemma}

\begin{proof}
From Tables \ref{chevbasson} and \ref{table:rootexpressionsD}, we have $x_{\alpha_i}(1) = I + E_{i,i+1} + E_{n+i+1,n+i}$ for all $1 \leq i \leq n-1$, where $I$ is the identity and $E_{i,j}$ is defined as before. Furthermore, we see that $x_{\beta_l}(1) = I + E_{2l-n-1,2n} + E_{n,2l-1}$. Now $u : = u_l$ acts on the basis elements of $V$ as follows: \begin{align*}
u (v_i) = & \ v_i + v_{i-1} + \cdots + v_1 \quad  (1 \leq i \leq n) & \\
u (v_j) = & \ v_j + v_{j+1} \quad (n+1 \leq j < 2n, \ j \neq 2l-1) & \\
u (v_{2l-1}) = & \ v_{2l-1} + v_{2l} + v_{n} + v_{n-1} + \cdots + v_1 & \\
u (v_{2n}) = & \ v_{2n} + v_{2l-n-1} + v_{2l-n-2} + \cdots + v_1. & 
\end{align*} 
A calculation shows that the fixed point space of $u$ has dimension $2$, and that it is spanned by $v_1$ and $v_{2l-n} + v_{2n}$. Therefore the Jordan normal form of $u$ has two Jordan blocks. To see that the Jordan block sizes are $2l-2$ and $2n-2l+2$, it suffices to show that $(u-1)^{2l-2} = 0$ and $(u-1)^{2l-3} \neq 0$, as then the largest Jordan block size in $u$ is $2l-2$. To this end, a calculation shows that $(u-1)^{2l-2}(v_i) = 0$ for all $i$ and $$(u-1)^{2l-3}(v_{n+1}) = \begin{cases} v_{1}, & \text{ if } 2l - 2 > n, \\ v_{2n} + v_2 + v_1, & \text{ if } 2l-2 = n . \end{cases}$$
It is now clear that $((u-1)^{2l-3}(v_{n+1}),v_{n+1}) =1$ and $(u-1)^{2l-2}(v_{n+1}) = 0$. Therefore $V(2l-2)$ occurs as summand of $V \downarrow k[u]$ by \cite[Lemma~6.9]{Mikko1}, and we must have $V \downarrow k[u] = V(2l-2) \oplus V(2n-2l+2)$, as required. 
\end{proof}

\begin{lemma} \label{l:Dnilprep}
The element $e_l = e_{\beta_l} + \sum_{i=1}^{n-1} e_{\alpha_i}$ acts on $V$ via $V\downarrow k[e_l]=W_l(n)$. 
\end{lemma}

\begin{proof}
The tables above yield $e_{\alpha_i} = E_{i,i+1}  + E_{n+i+1,n+i}$ for $i < n$ and $e_{\beta_l} = E_{2l-n-1,n} + E_{n,2l-1}$, so that $$e := e_l = E_{2l-n-1,2n} + E_{n,2l-1} + \sum_{i=1}^{n-1} (E_{i,i+1}  + E_{n+i+1,n+i}).$$ We need to show that $e$ has two Jordan blocks of size $n$ and that $l = \text{min}\{m \mid Q(e^m(v)) = 0 \ \forall \ v \in V\}$. We start by calculating the action of $e$ on the basis vectors of $V$:  
\begin{align*}
e (v_1) = & \ 0  \\
e (v_i) = & \ v_{i-1} \  (2 \leq i \leq n)  \\
e (v_j) = & \ v_{j+1} \  (n+1 \leq j < 2n, j \neq 2l - 1)  \\
e (v_{2l-1}) = & \ v_{2l} + v_{n} \\
e (v_{2n}) = & \ v_{2l-n-1}. 
\end{align*} 
One sees that the kernel of $e$ has dimension $2$, and it is spanned by $v_1$ and $v_{2l-n} + v_{2n}$. Therefore the Jordan normal form of $e$ has two Jordan blocks and a routine calculation with the  basis $\{v_i\}$ shows that $e^n(v) = 0$ for all $v\in V$. 

Finally, we must show that $l = \text{min}\{m \mid Q(e^m(v)) = 0 \ \forall \ v \in V\}$. Since $e^{l-1} (v_{n+1}) = v_{l} + v_{n+l}$, it follows that $Q(e^{l-1} (v_{m+1})) = 1$ and so $Q(e^{j} (e^{l-1-j} (v_{m+1}))) = 1$ for all $j \leq l-1$. It remains to prove that $Q(e^l(v)) =0$ for all $v \in V$. Since $2l \geq n$, it follows that $(e^l(v),e^l(w)) = (v,e^{2l}(w)) = 0$. Therefore, it suffices to show that $Q(e^l(v_i)) =0$ for all basis vectors $v_i$. If $1 \leq i \leq n$ or $2l \leq i \leq 2n$ then $e^{l}(v_i) = v_j$ for some $j$ and so $Q(e^{l}(v_i))=0$. When $n+1 \leq i \leq 2l-1$, we have $e^{l}(v_i) = v_{l+i} + v_{n-i+l}$ and so again, $Q(e^{l}(v_i))=0$. 
\end{proof}

\begin{remark}\label{remark:Duniprep}
Interpreting the coefficients of the unipotent element $u_l$ of Lemma \ref{l:Duniprep} as integers, reduction mod $p$ for $p>2$ gives an element acting on the natural module with Jordan block sizes $(2l-1,2n-2l+1)$. Indeed, a calculation shows that the fixed point space is spanned by $v_1$ and $v_{2l-n} + v_{2n}$, and furthermore $$(u-1)^{2l-2}(v_j) = \begin{cases} (-1)^n 2v_1, & \text{ if } j = n+1 \\ 0, &\text{ if } j \neq n+1\end{cases}$$ so the largest Jordan block size of $u_l$ is $2l-1$.

Therefore $u_l$ is regular in the subsystem $B_m B_{n-m+1}$ of $D_n$ and is therefore another representative alongside that provided in Lemma \ref{lemma:typeDoddsystem}. Similarly one can calculate that the nilpotent element $e_l$ of Lemma \ref{l:Dnilprep} has Jordan block sizes $(2l-1,2n-2l+1)$ in characteristic $p \neq 2$. Since the roots involved in $u_l$ and $e_l$ are $\Z$-independent, by \cite[Lemma 16.2C]{HumphreysGroupBook} these representatives have the advantage that they work regardless of the choice of signs of the structure constants.\end{remark}

\subsection{Exceptional types}

In this section we prove Theorem \ref{nilreps} for $G$ of exceptional type. If $x$ is unipotent, this follows from work of Lawther in \cite{L09}. It is routine to refine his lists to remove those distinguished unipotent classes of $G$ which are not eminent. For each non-eminent unipotent class we list its maximal subsystem overgroups in Table \ref{t:unipinexcep}, for completeness. We note that a representative of a unipotent class with label $X$ is simply $\prod_{j \in J} x_{\alpha_j}(1)$, where $\sum_{j \in J} e_{\alpha_j}$ is the representative of the nilpotent element with label $X$ given in \cite{Ste16}. This follows from the proofs in \cite{LS12} (see the introduction of Sections 17 and 18 in ibid.) since the representatives in \cite{Ste16} are deduced from \cite{LS12}.  

In good characteristic, \cite[p.24]{LT11} provides enough information to find the non-eminent distinguished nilpotent orbits. As might be expected, these classes are in bijection with the non-eminent distinguished unipotent classes (sending a unipotent class to the nilpotent orbit of the same label). It remains to consider the nilpotent orbits of $\Lie(G)$ in bad characteristic. We follow Lawther's approach, which entails constructing representatives of each distinguished nilpotent orbit of $\Lie(M)$ for $M$ a maximal subsystem subgroup of $G$ and determining its $G$-class. To do this we make heavy use of use Theorem~\ref{p:nilprec}.

Carrying out this calculation provides a proof of Theorem \ref{nilreps}. Lastly, we mention Table \ref{t:nilpinexcep}. It provides a list of distinguished, non-eminent nilpotent orbits, giving  their maximal subsystem overalgebras, for all characteristics. Specialising to good characteristic, we recover the analogous result in \cite{L09}. 

\begin{center}\begin{longtable}{|c|c|c|c|}\hline Type & Char. & Label & Maximal subsystem overgroups \\\hline
$G_2$ & any & $G_2(a_1)$ & $A_2$, $A_1 \tilde{A}_1$ $(p \neq 2)$, $\tilde{A}_2$ $(p=3)$ \\
\hline
$F_4$ & any & $F_4(a_1)$ & $B_4$, $C_4$ $(p=2)$ \\
& $p \neq 2$ & $F_4(a_2)$ & $A_1 C_3$ \\
& any & $F_4(a_3)$ & $B_4$, $A_2 \tilde{A}_2$ $(p \neq 3)$, $A_1 C_3$ $(p \neq 2)$, $C_4$ $(p=2)$ \\
& $p=2$ & $(C_3(a_1))_2$ & $B_4$, $C_4$ \\
\hline
$E_6$ & $p \neq 2$ & $E_6(a_3)$ & $A_1 A_5$  \\
\hline
$E_7$ & $p \neq 2$ & $E_7(a_3)$ & $A_1 D_6$ \\
& any & $E_7(a_4)$ & $A_1 D_6$, $A_7$ ($p=2$) \\
& any & $E_7(a_5)$ & $A_1 D_6$ $(p \neq 2)$, $A_2 A_5$ $(p \neq 3)$ \\
\hline
$E_8$ & $p \neq 2$ & $E_8(a_3)$ & $A_1 E_7$ \\
& $p \neq 2$ & $E_8(a_4)$ & $D_8$ \\
& $p \neq 2$ & $E_8(a_5)$ & $D_8$ \\
& any & $E_8(a_6)$ & $A_8$ $(p \neq 3)$, $D_8$ $(p \neq 2)$ \\
& any & $E_8(a_7)$ & $A_4 A_4$ $(p \neq 5)$, $A_2 E_6$ $(p \neq 3)$, $A_1 E_7$ $(p \neq 2)$, $D_8$ $(p \neq 2)$ \\
& any & $E_8(b_4)$ & $A_1 E_7$, $D_8$ $(p=2)$ \\
& any & $E_8(b_5)$ & $A_2 E_6$ $(p \neq 3)$, $A_1 E_7$ $(p \neq 2)$ \\
& any & $E_8(b_6)$ & $A_2 E_6$, $D_8$ $(p \neq 2)$, $A_8$ $(p=3)$ \\ 
& $p=2$ & $(D_7(a_1))_2$ & $A_1 E_7$, $D_8$ \\
\hline
\caption{Maximal subsystem overgroups of non-eminent distinguished unipotent elements in exceptional types}\label{t:unipinexcep}\end{longtable}\end{center}

\begin{center}\begin{longtable}{|c|c|c|c|}\hline Type & Char. & Label & Maximal subsystem overalgebras \\\hline
$G_2$ & any & $G_2(a_1)$ & $A_2$, $A_1 \tilde{A}_1$ $(p \neq 2)$ \\
\hline
$F_4$ & any & $F_4(a_1)$ & $B_4$ \\
& $p \neq 2$ & $F_4(a_2)$ & $A_1 C_3$ \\
& any & $F_4(a_3)$ & $B_4$, $A_2 \tilde{A}_2$ $(p \neq 3)$, $A_1 C_3$ $(p \neq 2)$ \\
& $p=2$ & $(C_3(a_1))_2$ & $B_4$ \\
\hline
$E_6$ & $p \neq 2$ & $E_6(a_3)$ & $A_1 A_5$  \\
\hline
$E_7$ & $p \neq 2$ & $E_7(a_3)$ & $A_1 D_6$ \\
& any & $E_7(a_4)$ & $A_1 D_6$ \\
& any & $E_7(a_5)$ & $A_1 D_6$ $(p \neq 2)$, $A_2 A_5$ $(p \neq 3)$ \\
\hline 
\endfirsthead
$E_8$ & $p \neq 2$ & $E_8(a_3)$ & $A_1 E_7$ \\
& $p \neq 2$ & $E_8(a_4)$ & $D_8$ \\
& $p \neq 2$ & $E_8(a_5)$ & $D_8$ \\
& any & $E_8(a_6)$ & $A_8$ $(p \neq 3)$, $D_8$ $(p \neq 2)$ \\
& any & $E_8(a_7)$ & $A_4 A_4$ $(p \neq 5)$, $A_2 E_6$ $(p \neq 3)$, $A_1 E_7$ $(p \neq 2)$, $D_8$ $(p \neq 2)$ \\
& any & $E_8(b_4)$ & $A_1 E_7$ \\
& any & $E_8(b_5)$ & $A_2 E_6$ $(p \neq 3)$, $A_1 E_7$ $(p \neq 2)$ \\
& any & $E_8(b_6)$ & $A_2 E_6$, $D_8$ $(p \neq 2)$ \\ 
\hline
\caption{Maximal subsystem overalgebras of non-eminent distinguished nilpotent elements in exceptional types}\label{t:nilpinexcep}\end{longtable}\end{center}

\section{Appendix: Auxiliary tables for Theorem \ref{p:nilprec}} \label{s:auxtabs}

Below follows some auxiliary data for Theorem \ref{p:nilprec}. Let $G$ be expectional. In each table we use a horizontal line to separate the sets of nilpotent orbits (represented by their labels) for which the Jordan block structure on $V_{\text{min}}$ (when it exists) and $\Lie(G)$ coincide. For each class we then provide the additional data required to distinguish them. See the introduction for the terminology (DS, ALG, etc.).

\begin{table}[h]
\begin{minipage}[b]{0.5\linewidth}\centering
\begin{tabular}{|c|l|}
\hline
Label & DS data \\
\hline
$F_4(a_1)$ & $10,5,2,0$ \\
$F_4(a_2)$ & $10,3,0$ \\
\hline
$B_3$ &  $14,9,6$ \\
$F_4(a_3)$ &  $14,7,3,0$ \\
\hline
$(C_3(a_1))_2$ &  $16,8,4,0$ \\
$(\tilde{A}_2 A_1)_2$ &  $16,8,3,0$ \\
\hline
$C_3(a_1)$ &  $20,16,15,7,4,0$ \\
$\tilde{A}_2 A_1$ &  $20,16,15,7,1,0$ \\
 \hline
\end{tabular}
\caption{$G = F_4$, $p=2$}
\end{minipage}%
\begin{minipage}[b]{0.5\linewidth}\centering
\begin{tabular}{|c|l|}
\hline
Label  & DS data \\
\hline
$D_5$ &  $12,8,6, 2,0$ \\
$E_6(a_3)$ &  $12,6,0$ \\
 \hline
$D_4$ &  $20,17,16$ \\
$D_4(a_1)$ &  $20,15,8,0$ \\
\hline
$A_3 A_1$ &  $24,21,20,14,12,4,0$ \\
$A_2^2 A_1$  & $24,21,20,12,1,0$ \\
\hline
\end{tabular} 
\caption{$G = E_6$, $p=2$}
\end{minipage}
\end{table}

\begin{table}[h]
\begin{minipage}[b]{0.5\linewidth}\centering
\begin{tabular}{|c|l|}
\hline
Label & ALG data  \\
\hline
$E_6$ & $31$ (sc) / $24$ (ad) \\
$E_6(a_1)$ & $60$ (sc) / $29$ (ad) \\
\hline
\end{tabular}
\caption{$G = E_6$, $p=3$}
\end{minipage}%
\begin{minipage}[b]{0.5\linewidth}\centering
\begin{tabular}{|c|l|}
\hline
Label & DS data \\
\hline
$E_6$ &  $15,12,9$ \\
$E_6(a_1)$ &  $15,9,4,0$ \\
\hline
$D_6(a_1)$ &  $19,15$ \\
$A_6$ & $19,17$ \\
\hline
\end{tabular} 
\caption{$G = E_7$, $p=3$}
\end{minipage}
\end{table}

\begin{table}[h]
\begin{minipage}[b]{0.3\linewidth}\centering
\begin{tabular}[t]{|c|l|}
\hline
 Label & ALG data \\
 \hline
 $E_7$ & 62 \\
 $E_7(a_1)$ & 67 \\
 \hline
 $D_{6}$ & 58 \\
 $E_{7}(a_{4})$ & 63 \\
  $D_{6}(a_{1})$ & 70 \\
 $(A_{6})_{2}$ & 75 \\
 \hline
 $D_{5} A_{1}$ & 73 \\
 $E_{7}(a_{5})$ & 91 \\
 \hline
$D_{5}$ & 70 \\
 $E_{6}(a_{3})$ & 84 \\
 \hline
 $D_{4} A_{1}$ & 63 \\
 $A_{3} A_{2} A_{1}$ & 69 \\
 $(A_{3} A_{2})_{2}$ & 74 \\ 
 $D_{4}(a_{1}) A_{1}$ & 83 \\
 \hline
$D_{4}$ & 66 \\
 $D_{4}(a_{1})$ & 82 \\
 \hline
 \end{tabular}
 \end{minipage}%
\begin{minipage}[b]{0.3\linewidth}\centering
 \begin{tabular}[t]{|c|l|}
 \hline
Label & DS data \\
\hline
 $D_{6}(a_{2})$ & $26,21,20,12,2,0$ \\
 $A_{5} A_{1}$ & $26,21,20,12,1,0$ \\
\hline
$(A_{3} A_{1})^{(1)}$ & $43, 39, 37, 29, 21, 20, 4, 0$ \\
$A_{2}^2 A_{1}$ & $43, 39, 37, 27, 8, 0$ \\
\hline
\end{tabular}
\end{minipage}
\caption{$G = E_7$ (simply connected), $p=2$}
\end{table}

\begin{table}[h]
\begin{minipage}[b]{0.3\linewidth}\centering
\begin{tabular}[t]{|c|l|}
\hline
Label & ALG data \\
\hline
$E_{7}$ & $56$ \\
$E_{7}(a_{1})$ & $58$ \\
\hline
$D_{6}$ & $52$ \\
$E_{7}(a_{4})$ & $54$ \\
$D_{6}(a_{1})$ & $56$ \\
$(A_{6})_{2}$ & $57$ \\
\hline
$D_{5}A_{1}$ & $60$ \\
$E_{7}(a_{5})$ & $71$ \\
\hline
$D_{5}$ & $42$ \\
$E_{6}(a_{3})$ & $50$ \\
\hline
$A_{4}A_{2}$ & $54$ \\
$D_{5}(a_{1})$ & $39$ \\
\hline
$D_{4}A_{1}$ & $57$ \\
$A_{3}A_{2}A_{1}$ & $58$ \\
$(A_{3}A_{2})_{2}$ & $60$ \\
$D_{4}(a_{1})A_{1}$ & $65$ \\
\hline
$A_{2}^2$ & $56$ \\
$A_{3}$ & $55$ \\
\hline
\end{tabular}
\end{minipage}%
\begin{minipage}[b]{0.3\linewidth}\centering
\begin{tabular}[t]{|c|l|}
\hline
Label & DS data \\
\hline
$D_{6}(a_{2})$ & $26, 21, 20, 12, 2, 0$ \\
$A_{5}A_{1}$ & $26, 21, 20, 12, 1, 0$ \\
\hline
$A_{3}A_{2}$ & $39, 35, 26, 23, 7, 0$ \\
$D_{4}$ & $39, 35, 28$ \\
$D_{4}(a_{1})$ & $39, 35, 28, 26, 14, 0$ \\
\hline
$(A_{3}A_{1})^{(1)}$ & $43, 40, 37, 29, 21, 20, 4, 0$ \\
$A_{2}^2A_{1}$ & $43, 40, 37, 27, 8, 0$ \\
\hline
\end{tabular}
\end{minipage}
\caption{$G = E_7$ (adjoint), $p=2$}
\end{table}

\begin{table}[h]
\begin{minipage}[b]{0.3\linewidth}\centering
\begin{tabular}{|c|c|l|}
\hline
Label & DS data \\
\hline
$E_8$ & $10,7,0$ \\
$E_8(a_1)$ & $10,1,0$ \\
\hline
\end{tabular}
\caption{$G = E_8$, $p=5$}
\end{minipage}%
\begin{minipage}[b]{0.7\linewidth}\centering
\begin{minipage}[b]{0.5\linewidth}
\begin{tabular}[t]{|c|l|}
\hline
Label & NDS data  \\
\hline
$E_{8}$ &  $15, 18, 248$ \\
$E_{8}(a_{1})$ & $16, 18, 248$ \\
\hline
$E_{6} A_{1}$ & $32, 33, 39$ \\
$E_{8}(b_{6})$ & $33, 34, 39, 60, 165, 248$ \\
$E_{6}(a_{1}) A_{1}$ & $34, 40, 53, 132, 248$ \\
$(A_{7})_{3}$ & $33, 35, 47, 105, 248$ \\
\hline
$A_{7}$ & $36, 36, 36$ \\
$E_{6}$ & $36, 37, 59$ \\
$E_{6}(a_{1})$ & $39, 56, 60$ \\
\hline
$E_{7}(a_{4})$ & $40, 47, 63$ \\
$A_{6} A_{1}$ & $38, 38, 38$ \\
\hline
$D_{6}(a_{1})$ & $43, 49$ \\
$A_{6}$ & $41, 41$ \\
\hline 
\end{tabular}
\end{minipage}%
\begin{minipage}[b]{0.3\linewidth}
\begin{tabular}[t]{|c|l|}
\hline
Label & NIL data  \\
\hline
$E_{6}(a_{3}) A_{1}$ & $36$ \\
$D_{5}(a_{1}) A_{2}$ & $43$ \\
\hline
\end{tabular}
\end{minipage}
\caption{$G = E_8$, $p=3$}\end{minipage}\end{table}

\begin{table}[h]
\begin{minipage}[b]{0.26\linewidth}\centering
\begin{tabular}[t]{|c|l|}
\hline
Label & ALG, ALG' \\
& data \\
\hline
$E_{8}$ & $79,74$ \\

$E_{8}(a_{1})$ & $79,73$ \\

$E_{8}(a_{2})$ & $89,85$ \\

$E_{8}(a_{3})$ & $108,103$ \\

$E_{8}(a_{4})$ & $139,136$ \\
\hline

$E_{7}$ & $68,64$ \\

$E_{8}(b_{4})$ & $70,63$ \\

$E_{7}(a_{1})$ & $72,69$ \\
\hline

$E_{8}(a_{5})$ & $79,72$ \\

$E_{8}(b_{5})$ & $91,86$ \\
\hline
$(D_{7})_{2}$ & $102,98$ \\

$E_{8}(a_{6})$ & $107,101$ \\
\hline
$D_{7}$ & $69,65$ \\

$D_{7}(a_{1})$ & $77,70$ \\

$A_{7}$ & $85,72$ \\

$D_{7}(a_{2})$ & $101,93$ \\
\hline
$E_{7}(a_{2})$ & $61,57$ \\

$E_{6} A_{1}$ & $61,53$ \\
\hline
$(D_{7}(a_{1}))_{2}$ & $69,63$ \\

$E_{7}(a_{3})$ & $82,78$ \\
\hline
$D_{6}$ & $71,68$ \\

$(D_{5} A_{2})_{2}$ & $72,65$ \\

$E_{7}(a_{4})$ & $74,67$ \\

$A_{6} A_{1}$ & $77,66$ \\

$D_{6}(a_{1})$ & $82,78$ \\

$(A_{6})_{2}$ & $85,76$ \\
\hline
\end{tabular}
\end{minipage}%
\begin{minipage}[b]{0.26\textwidth}\centering
\begin{tabular}[t]{|c|l|}
\hline
Label & ALG, ALG' \\
& data \\
\hline
$D_{5} A_{2}$ & $76,67$ \\

$E_{8}(a_{7})$ & $102,99$ \\
\hline
$D_{6}(a_{2})$ & $73,71$ \\

$A_{5} A_{1}$ & $73,58$ \\
\hline
$D_{5}$ & $64,60$ \\

$E_{6}(a_{3})$ & $70,64$ \\
\hline
$(D_{4} A_{2})_{2}$ & $64,58$ \\

$A_{4} A_{2} A_{1}$ & $64,57$ \\

$D_{5}(a_{1}) A_{1}$ & $71,64$ \\
\hline
$D_{4} A_{2}$ & $83,73$ \\

$A_{3}^2$ & $88,76$ \\

$D_{4}(a_{1}) A_{2}$ & $99,72$ \\
\hline
$D_{4} A_{1}$ & $91,88$ \\

$A_{3} A_{2} A_{1}$ & $92,86$ \\

$(A_{3} A_{2})_{2}$ & $101,95$ \\

$D_{4}(a_{1}) A_{1}$ & $109,104$ \\
\hline
$D_{4}$ & $88,85$ \\

$D_{4}(a_{1})$ & $92,90$ \\
\hline
$A_{3} A_{1}$ & $94,92$ \\

$2A_{2} A_{1}$ & $94,87$ \\
\hline 
\end{tabular}
\end{minipage}%
\begin{minipage}[b]{0.43\textwidth}\centering
\begin{tabular}[t]{|c|l|}
\hline
Label & DS data \\
\hline
$D_{5} A_{1}$ & $44,39,37,29,19,18,2,0$ \\

$E_{7}(a_{5})$ & $44,39,37,29,17,1,0$ \\

$E_{6}(a_{3}) A_{1}$ & $44,39,37,27,8,0$ \\

\hline
$D_{5}(a_{1}) A_{2}$ & $48,43,37,21,1,0$ \\

$A_{4} A_{3}$ & $48,43,37,19,1,0$ \\
\hline

$A_{3} A_{1}^2$ & $80,76,75,71,70,54,39,38,6,0$ \\

$A_{2}^2 A_{1}^2$ & $80,76,75,71,70,50,14,0$ \\

\hline
\end{tabular}
\end{minipage}
\caption{$G = E_8$, $p=2$}
\end{table}

\clearpage

{\footnotesize
\bibliographystyle{amsalpha}
\providecommand{\bysame}{\leavevmode\hbox to3em{\hrulefill}\thinspace}
\providecommand{\MR}{\relax\ifhmode\unskip\space\fi MR }
\providecommand{\MRhref}[2]{%
  \href{http://www.ams.org/mathscinet-getitem?mr=#1}{#2}
}
\providecommand{\href}[2]{#2}

}

\end{document}